\documentclass[a4paper,11pt]{article}
\usepackage[utf8]{inputenc}

\usepackage{amsmath}
\usepackage{amsthm}
\usepackage{amssymb}

\newtheorem{definition}{Definition}[section]
\newtheorem{lemma}[definition]{Lemma}
\newtheorem{theorem}[definition]{Theorem}

\title{\bf Unsteady 3D-Navier-Stokes System with Tresca's Friction Law}

\author{ 
 {Mahdi Boukrouche\thanks{Corresponding author Phone +33 4 77 48 15 35,
  E-mail: Mahdi.Boukrouche@univ-st-etienne.fr}, 
  Imane Boussetouan
  and  Laetitia Paoli
  }
\thanks{Address for all : Lyon University, F-42023 Saint-Etienne,
Institut Camille Jordan CNRS UMR 5208,
23 rue  Paul Michelon 42023 Saint-Etienne Cedex 2, France. 
  imane.boussetouan@univ-st-etienne.fr,
 laetitia.paoli@univ-st-etienne.fr}
 }

\date{}

\begin{document}

\maketitle\normalsize

\begin{abstract}
Motivated by extrusion problems, we consider a non-stationary incompressible 3D fluid flow with a
non-constant (temperature dependent) viscosity, subjected to mixed boundary conditions with a given
time dependent velocity on a part of the boundary and Tresca's friction law on the other part.
 We construct a sequence of approximate solutions by using a regularization of the free boundary
condition due to friction combined with a particular penalty method, reminiscent of the ``incompressibility limit'' of compressible fluids, allowing to get better insights
 into the links between the fluid velocity and pressure fields. Then we pass to the limit with
compactness arguments to obtain a solution to our original problem.
\end{abstract}


\section{Introduction}\label{Introduction}

Fluid flow problems are involved in several physical phenomena and play an important role in many
 industrial applications. The fundamental model in fluid mechanics is the well-known Navier-Stokes
 system for incompressible viscous fluids which has been intensively studied during the last 78 years.
 Since the pioneering work of J.~Leray \cite{leray} in 1934, the mathematical analysis of this
problem has performed significant progresses: we can mention here only few selected references
 \cite{lions2, girault, temam, boyer, foias, galdi}. Nevertheless it is still a very active research
field, from both the theoretical point of view and the numerical point of view (see for instance
the very recent research articles \cite{bemelmans, physicaD1, physicaD2, zhang}).

Motivated by   extrusion problems 
we consider in this paper a non-stationary incompressible  3D fluid flow
with a temperature dependent viscosity. As usual for this kind of problems the extrusion device is
composed of an upper fixed part and a  lower moving part. Several experiments have shown that the
classical adhesion condition between the fluid and the  lower moving part of the boundary of its domain
 is not satisfied and  the real behavior seems to be governed by some friction condition of Tresca's type
\cite{hervet}  \cite{baffico}.

More precisely, let $\omega$ be a non empty open bounded subset with a  Lipschitz continuous
boundary,  of $\mathbb{R}^{d-1}$ for $d=2, 3$.
We denote by $\Omega\subset \mathbb{R}^{d}$ the domain of the flow  given by
\begin{eqnarray*}
\Omega = \bigl\{(x',x_{d})\in \mathbb{R}^{d}:\quad x'\in\omega,\quad 0< x_{d}< h(x')\bigr\},
\end{eqnarray*}
where $x'= (x_{1},\ldots,x_{d-1})\in \mathbb{R}^{d-1}$, $x=(x',x_{d})\in \mathbb{R}^{d}$.
 The boundary of $\Omega$ is $\partial\Omega=\Gamma_0 \cup \Gamma_{L}\cup \Gamma_{1}$,
where
$\Gamma_0=\{(x',x_{d})\in\overline{\Omega}: x_{d}=0\}$, $\Gamma_{1}=\{(x',x_{d})\in\overline{\Omega}:
 x_{d}=h(x')\}$ and $\Gamma_{L}$ is the lateral boundary. We assume that $h$ is a  Lipschitz continuous function 
 and there exist two real numbers $h_{min}$ and  $h_{max}$ such that $0< h_{min}< h(x')< h_{max}$
 for all $x'\in\mathbb{R}^{d-1}$.

Let us emphasize that we do not introduce any restrictive assumption on the thickness of the domain.
On the contrary to previous papers where only thin films where studied \cite{rao, bouk1, bouk2},
 we can consider here general 3D geometries.

The fluid flow is described by the conservation of momentum
\begin{eqnarray*}
\frac{\partial v}{\partial t}+ (v.\nabla) v= div(\sigma)+ f\quad\mbox{in}\quad\Omega\times(0,\tau),
\end{eqnarray*}
and  the incompressibility  condition
\begin{eqnarray*}
div(v) = 0 \quad\mbox{in}\quad\Omega\times(0,\tau),
\end{eqnarray*}
where $v$ is the velocity field of the fluid flow,
$f$ represents the density of body forces and $\sigma$ is the stress tensor. We assume that the fluid is Newtonian, so
\begin{eqnarray*}
\sigma = -pI + 2\mu(T)D(v),
\end{eqnarray*}
where $T$ depending on $(x,t) \in \Omega \times (0,\tau)$, is the temperature field.             
Note that $T$ stands for temperature but it will not appear as a variable
of the problem, the time interval on which the equations are considered is :
$[0 , \tau]$. We do this for the main reason that we have  generalised our work to a
coupled problem (velocity-pressure-temperature)  which is in final version,
so here we give the regularity of $T$ suitable for our coupled problem \cite{MLIC}.
 $\mu (T)$ is the temperature dependent viscosity of the fluid, $p$ is the pressure and   $D(v)$ is the strain rate  tensor given by
$$D(v)= \bigl( d_{ij}(v) \bigr)_{1 \le i,j \le d}, \quad d_{ij}(v)=\frac{1}{2}\left(\frac{\partial v_{i}}{\partial x_{j}}+\frac{\partial v_{j}}{\partial x_{i}} \right)\, \quad 1\leq i,j\leq d.$$
Hence $v$ and $p$ satisfy the Navier-Stokes system
\begin{eqnarray}
\label{NS2}
\frac{\partial v}{\partial t}+(v.\nabla)v-2div \bigl(\mu (T) D(v) \bigr)+\nabla p
=f \quad\mbox{in}\quad\Omega\times(0,\tau),
\end{eqnarray}
\begin{eqnarray}
\label{NS9}
div (v)=0 \quad\mbox{in}\quad\Omega\times(0,\tau),
\end{eqnarray}
with the initial condition
\begin{eqnarray}\label{NS10}
v(0, \cdot) = v_{0}\quad\mbox{in}\quad  \Omega.
\end{eqnarray}

Let us now describe the boundary conditions. We denote by $s: \Gamma_0 \to \mathbb{R}^{d-1}$
the shear velocity of the lower surface of the extrusion device at $t=0$ and by $s \zeta (t)$,
with $\zeta: [0, \tau] \to \mathbb{R}$ such that $\zeta(0)=1$, its velocity at any instant $t \in [0, \tau]$.
 We introduce a  function $g: \partial \Omega \to \mathbb{R}^d$ such that
\begin{eqnarray*}
\int_{\Gamma_{L}}g. n \,d\sigma =0, \quad g=0 \  \mbox{on}\  \Gamma_1 ,
\quad g_{n} = g \cdot n = 0 \  \mbox{and}\  g_{\tau} = g- g_{n} n =s \  \mbox{on}\  \Gamma_0,
\end{eqnarray*}
where $n=(n_{1},\ldots, n_{d})$ is the unit outward normal  vector to  $\partial \Omega$.
We denote here by $u \cdot w$ the Euclidean inner product of two vectors $u$ and $w$  and by $|.|$ the Euclidian norm. We define respectively the normal and the tangential velocities on $\Gamma_0$ by
\begin{eqnarray*}
v_{n}=v\cdot n = v_{i} n_{i},\quad v_{{\cal T}} = \bigl( v_{{\cal T}i} \bigr)_{1 \le i \le d} \  \mbox{with} \  \ v_{{\cal T}i}=v_i -v_{n} n_i  \quad  1 \le i \le d
\end{eqnarray*}
and the normal and the tangential components of the stress tensor on $\Gamma_0$ by
\begin{eqnarray*}
\sigma_{n} = (\sigma\cdot n)\cdot n = \sigma_{ij} n_{i} n_{j},\quad \sigma_{{\cal T}} = \bigl( \sigma_{{\cal T } i}\bigr)_{1 \le i \le d}
\  \mbox{with} \    \ \sigma_{{\cal T} i}=\sigma_{ij} n_{j}-\sigma_{n} n_{i}  \quad 1 \le i \le d.
\end{eqnarray*}
Note that we will use  the Einstein's summation convention throughout this paper.
We assume that the upper surface of the extrusion device is fixed i.e.
\begin{eqnarray}\label{NS11}
v = 0\quad\mbox{on}\quad\Gamma_{1} \times (0, \tau),
\end{eqnarray}
the given velocity on the lateral boundary is the product $g(x)\zeta(t)$ i.e.
\begin{eqnarray}\label{NS11L1}
v=g\zeta\quad\mbox{on}\quad\Gamma_{L} \times (0, \tau),
\end{eqnarray}
and the normal component of the velocity on the lower part of  boundary is given by
\begin{eqnarray}\label{NS13}
v_{n}=v\cdot n =0\quad\mbox{on}\quad\Gamma_0 \times (0, \tau).
\end{eqnarray}
The tangential velocity on $\Gamma_0 \times (0, \tau)$ is unknown and satisfies  Tresca's friction law \cite{duvaut1}
\begin{equation}\label{NS14}
\begin{array}{l}
\displaystyle{|\sigma_{{\cal T}}|< \ell \Rightarrow v_{{\cal T}} = (s\zeta, 0) }\\
\displaystyle{|\sigma_{{\cal T}}| = \ell \Rightarrow \exists\lambda\geq0\quad\mbox{such that}\quad v_{{\cal T}}=(s\zeta, 0) -\lambda\sigma_{{\cal T}}}
\end{array}
\end{equation}
where $\ell : [0 , \tau ]\times \Gamma_0 \to  \mathbb{R}$ is the upper limit for the shear stress (i.e. $\ell$ is the Tresca's friction threshold).

\smallskip

The paper is organized as follows. In Section \ref{formulation} we introduce the functional framework
and the formulation of the problem as a variational inequality for the fluid velocity and pressure fields.
 In Section \ref{approximate_problems} we use a regularization of  Tresca's functional to obtain a sequence
 of approximate problems $(P_{\varepsilon})$ of Navier-Stokes type.

A classical technique to study such problems is to choose divergence free test-functions in order to
 ``kill'' the pressure terms then to solve the derived variational problem for the fluid velocity and
to get finally the pressure by applying abstract results of functional analysis
(see \cite{temam, simon99, galdi} for instance).
The major drawback of this technique is that the pressure appears as a by product. In order to get better
insights into the links between the velocity and pressure fields, we adopt in this paper another approach, reminiscent of the ``incompressiblity limit'' of compressible fluids. 

More precisely, following an idea of J.L.~Lions \cite{lions}, we relax the divergence free condition and
 we propose a sequence of penalized problems $(P_{\varepsilon}^{\delta})$.
In Section \ref{penalized_problem} we establish the existence  of  solutions to this family of problems
 $(P_{\varepsilon}^{\delta})_{\varepsilon>0, \delta >0}$ and we obtain some a priori estimates.
Next in Section \ref{approximate_pressure} we define a sequence of approximate pressures
$(p_{\varepsilon}^{\delta})_{\varepsilon>0, \delta >0}$ and we study its properties. By using functional spaces that are weaker in time than in space, we succeed in obtaining good enough  uniform estimates with respect to the   the parameters $\delta$ and $\varepsilon$.
Then in Section \ref{fin} we use compactness arguments to pass to the incompressible limit as $\delta$ tends to zero
 and we show that the limit velocity and pressure fields are solutions to the problems $(P_{\varepsilon})$.
 Finally  we pass to the limit as $\varepsilon$ tends to zero and we get a solution to our original
 variational inequality.

\section{Variational formulation of the problem}
\label{formulation}

We denote by
\begin{eqnarray*}
{\bf H}^{1}(\Omega) = \left(H^{1}(\Omega)\right)^{d},\
{\bf L}^{2}(\Omega) = (L^{2}(\Omega))^{d},\
{\bf H}^{\frac{1}{2}}(\partial\Omega) = (H^{\frac{1}{2}}(\partial\Omega))^{d},\ {\bf H}^{2}(\Omega) = \left(H^{2}(\Omega)\right)^{d}. 
\end{eqnarray*}
We assume  that
\begin{eqnarray}\label{NS14L}
\begin{array}{ll}
\displaystyle{f\in L^{2}\bigl(0,\tau;{\bf L}^{2}(\Omega) \bigr), \  \ell \in L^{2}\bigl(0,\tau; {\bf L}^{2}_+(\Gamma_0) \bigr) \cap L^{\infty}\bigl(0,\tau; {\bf L}^{\infty}_+(\Gamma_0) \bigr) ,}\\
\displaystyle{  \zeta \in {\cal C}^{\infty} \bigl( [0, \tau] \bigr) \ \mbox{with}\ \zeta (0)=1}, 
\end{array}
\end{eqnarray}
 with
${\bf L}^{2}_+(\Gamma_0)  = {\bf L}^{2}(\Gamma_0; \mathbb{R}^+ )$  (respectively  
${\bf L}^{\infty}_+(\Gamma_0)  = {\bf L}^{\infty}(\Gamma_0; \mathbb{R}^+ )$  and $\tau>0$.
The viscosity  $\mu (T)$ is a function of $L^{\infty} \bigl( 0, \tau; L^{\infty}(\Omega) \bigr)$
depending on the temperature $T$, and there exist two real numbers  $\mu^{*}$, $\mu_{*}$ such that
\begin{eqnarray}\label{NS20}
0< \mu^{*} \leq 2 \mu(X)\leq \mu_{*}\quad\forall X\in\mathbb{R},
\end{eqnarray}
and also  there exists an extension of $g$ to $\Omega$, denoted by $G_0$, such that
\begin{eqnarray}\label{G0HYPO}
\begin{array}{ll}
\displaystyle{G_{0}\in {\bf H}^{2}(\Omega), \  div  (G_{0})=0\ \mbox{in}\ \Omega,\  G_{0}=g\ \mbox{on}\ \Gamma_{L}, \
G_{0}=0\ \mbox{on}\ \Gamma_{1},}\\
\displaystyle{  G_{0 n} =0  \  \mbox{and}\  G_{0 \tau}=s\ \mbox{on}\  \Gamma_0.}
\end{array}
\end{eqnarray}
We introduce now the following functional framework
\begin{eqnarray*}
{\cal V}_{0} = \left\{\varphi\in {\bf H}^{1}(\Omega):\  \varphi=0\ \mbox{on}\  \Gamma_{L}\cup\Gamma_{1}, \ \varphi_n=0\ \mbox{on}\ \Gamma_0\right\},
\end{eqnarray*}
endowed  with the norm of ${\bf H}^{1}(\Omega)$ and
\begin{eqnarray*}
{\cal V}_{0div} = \bigl\{\varphi\in {\cal V}_0:\ div (\varphi) =0\ \mbox{in}\ \Omega\bigr\}.
\end{eqnarray*}
Moreover let
\begin{eqnarray*}
L^{2}_{0}(\Omega) = \left\{q\in  L^{2}(\Omega):\ \int_{\Omega} q\,dx = 0\right\}.
\end{eqnarray*}
 We define the following applications
\begin{eqnarray*}
a (T; \cdot , \cdot) : L^2 (0, \tau; {\cal V}_{0}) \times  L^2(0, \tau; {\cal V}_{0}) &\rightarrow& \mathbb{R}\\
(u,v)&\mapsto& a(T;u,v)=
 \int_0^\tau \int_{\Omega} 2 \mu(T) D(u) : D(v) \, dx  dt \\
& \qquad & \qquad \qquad  =
\int_0^\tau \int_{\Omega} 2 \mu(T) d_{ij}(u)d_{ij}(v)\,dx  dt,
\end{eqnarray*}
and
\begin{eqnarray*}
\Psi :\quad L^2 \bigl(0,\tau; {\bf L}^2 (\Gamma_0) \bigr) &\rightarrow& \mathbb{R}\\
u&\mapsto& \Psi(u)= \int_0^\tau \int_{\Gamma_0} \ell |u|\,dx' dt .
\end{eqnarray*}
We may observe that  $\Psi$ is convex continuous but not differentiable.

Let $b$ be the usual trilinear form given by
\begin{eqnarray*}
b:\quad {\cal V}_{0}\times {\cal V}_{0}\times {\cal V}_{0} &\rightarrow& \mathbb{R}\\
(u,v,w)&\mapsto& b(u,v,w)=\int_{\Omega}u_{i}\frac{\partial v_{j}}{\partial x_{i}}w_{j}\,dx.
\end{eqnarray*}
By definition of ${\cal V}_{0}$ we have the identity
\begin{eqnarray} \label{formuleb}
b(u,v,w) = - b(u, w,v) - \int_{\Omega} div(u) v \cdot w \, dx  \quad \forall (u,v,w) \in {\cal V}_{0} \times {\cal V}_{0} \times {\cal V}_{0}.
\end{eqnarray}
Moreover, using Korn's inequality \cite{Korn} and assumption (\ref{NS20}), there exists $\alpha>0$ such that, for almost every $t \in (0, \tau)$, we have
\begin{eqnarray}\label{coercif}
\alpha \|u\|^2_{ {\bf H}^1 (\Omega)} \le \int_{\Omega} 2 \mu(T) D(u):D(u) \, dx  \le  \mu_* \|u\|^2_{ {\bf H}^1 (\Omega)} \quad \forall u \in  {\cal V}_{0}.
\end{eqnarray}

In order to deal with homogeneous boundary conditions on $\Gamma_L \cup \Gamma_1$, we set $\widetilde{v}=v-G_{0}\zeta$.
The variational formulation of the problem (\ref{NS2})-(\ref{NS14}) is given by (see for example \cite{duvaut1} and \cite{bouk1, bouk2})

\smallskip

\noindent {\bf Problem $(P)$}
Find
$$\widetilde{v}\in L^{2} \bigl(0,\tau;{\cal V}_{0div} \bigr)\cap L^{\infty}\bigl(0,\tau;{\bf L}^{2}(\Omega) \bigr),\  \,
\frac{\partial \widetilde{v}}{\partial t} \in L^{\frac{4}{3}} \bigl(0,\tau;({\cal V}_{0div})' \bigr) ,\  \,
p\in H^{-1}\bigl(0,\tau;L^{2}_{0}(\Omega) \bigr)$$
such that, for all $ \varphi\in {\cal V}_{0}$ and for all $\chi\in {\cal D}(0,\tau)$, we have
\begin{equation}\label{NS-25}
\begin{array}{ll}
\displaystyle
\left\langle \frac{d}{dt} \left( \widetilde{v}, \varphi \right) , \chi\right\rangle_{{\cal D}'(0,\tau), {\cal D}(0,\tau)}  +
 \bigl\langle b(\widetilde{v},\widetilde{v},\varphi) ,\chi \bigr\rangle_{{\cal D}'(0,\tau), {\cal D}(0,\tau)} - \bigl\langle \bigl(p,div(\varphi)  \bigr), \chi \bigr\rangle_{{\cal D}'(0,\tau), {\cal D}(0,\tau)} \\
\displaystyle + a(T;\widetilde{v},\varphi\chi)+\Psi(\widetilde{v}+ \varphi \chi )-\Psi(\widetilde{v})
\displaystyle  \geq
\bigl\langle (f,\varphi ), \chi \bigr\rangle_{{\cal D}'(0,\tau), {\cal D}(0,\tau)} - a(T;G_{0}\zeta ,\varphi \chi) \\
\displaystyle - \left\langle \left(G_{0}\frac{\partial \zeta}{\partial t},\varphi \right),
\chi \right\rangle_{{\cal D}'(0,\tau), {\cal D}(0,\tau)} -  \bigl\langle b(G_{0}\zeta ,\widetilde{v}+G_{0}\zeta,\varphi),
\chi \bigr\rangle_{{\cal D}'(0,\tau), {\cal D}(0,\tau)} \\
\displaystyle - \bigl\langle b(\widetilde{v},G_{0}\zeta ,\varphi) , \chi \bigr\rangle_{{\cal D}'(0,\tau), {\cal D}(0,\tau)}
\end{array}
\end{equation}
with the initial condition
\begin{eqnarray}\label{NS-25init}
\widetilde{v}(0, \cdot) = \widetilde{v}_{0} \in  {\bf H},
\end{eqnarray}
where ${\bf H}$ is the well known closure in ${\bf L}^{2}(\Omega)$ of the space 
$$\{ \varphi \in {\cal C}^{\infty}(\overline{\Omega}) : \, div \, \varphi= 0  \mbox{ in }\Omega\},$$
$(\cdot, \cdot)$ denotes the inner product in ${\bf L}^2(\Omega)$ and $\langle \cdot, \cdot \rangle_{{\cal D}'(0,\tau), {\cal D}(0,\tau)}$ the duality product between ${\cal D}'(0,\tau)$ and ${\cal D}(0, \tau)$. Let us emphasize that we identify $\widetilde{v} + \varphi \chi $ and $\widetilde{v} $ with their trace on $\Gamma_0$ in the definition of $\Psi(\widetilde{v}+ \varphi \chi )$ and $\Psi(\widetilde{v})$.

\section{Approximate problems}
\label{approximate_problems}

The variational formulation of the problem (\ref{NS2})-(\ref{NS14}) leads to an inequality involving  Tresca's functional $\Psi$, which is convex continuous but not differentiable.
To overcome this  difficulty we use a regularization of $\Psi$. More precisley, for any $\varepsilon>0$, we introduce  $\Psi_{\varepsilon}$ defined by
\begin{eqnarray*}
 \Psi_{\varepsilon}(u)= \int_0^\tau \int_{\Gamma_0} \ell \sqrt{\varepsilon^{2} + |u|^{2}}\,dx' dt  \quad \forall u \in
 L^2 \bigl(0,\tau; {\bf L}^2(\Gamma_0) \bigr)
\end{eqnarray*}
which is G\^ateaux-differentiable in $L^2 \bigl(0,\tau; {\bf L}^2(\Gamma_0) \bigr) $,
with $\Psi'_{\varepsilon}(u) \in \bigl( L^2 \bigl(0,\tau; {\bf L}^2(\Gamma_0) \bigr) \bigr)'= L^2 \bigl(0,\tau; {\bf L}^2(\Gamma_0) \bigr) $
for all $u \in L^2 \bigl(0,\tau; {\bf L}^2(\Gamma_0) \bigr) $
given by
\begin{eqnarray*}
\langle \Psi'_{\varepsilon}(u),w \rangle = \int_0^{\tau} \int_{\Gamma_0} \ell \frac{u \cdot w}{\sqrt{\varepsilon^{2} + |u|^{2}}}\,dx' dt \quad \forall  w \in L^2 \bigl(0,\tau; {\bf L}^2(\Gamma_0) \bigr)
\end{eqnarray*}
where $\langle \cdot, \cdot \rangle$ denotes the inner product in $L^2 \bigl(0,\tau; {\bf L}^2(\Gamma_0) \bigr) $.
Then we consider a sequence of initial data $(\widetilde{v}_{{\varepsilon} 0})_{\varepsilon>0}$ such that
\begin{eqnarray} \label{init1}
\displaystyle \widetilde{v}_{{\varepsilon} 0} \longrightarrow_{\varepsilon \to 0} \widetilde{v}_{0} \quad \hbox{\rm strongly  in {\bf H}}
\end{eqnarray}
and we approximate problem $(P)$ by the following problems $(P_{\varepsilon})$, $\varepsilon >0$:

\smallskip

\noindent {\bf Problem $(P_{\varepsilon})$}
Find
$$\widetilde{v}_{\varepsilon}\in L^{2} \bigl(0,\tau;{\cal V}_{0div} \bigr)\cap L^{\infty}\bigl(0,\tau; {\bf L}^{2}(\Omega) \bigr),\  \, 
\frac{\partial \widetilde{v_{\varepsilon}}}{\partial t}\in L^{\frac{4}{3}}\bigl(0,\tau;({\cal V}_{0div})' \bigr),\ \,
p_{\varepsilon} \in H^{-1}\bigl(0,\tau;L^{2}_{0}(\Omega) \bigr)$$
such that, for all $ \varphi\in {\cal V}_{0}$ and for all $\chi\in {\cal D}(0,\tau)$, we have
\begin{equation}\label{NS-25bis}
\begin{array}{ll}
\displaystyle
\left\langle \frac{d}{dt} \left( \widetilde{v}_{\varepsilon}, \varphi \right) , \chi\right\rangle_{{\cal D}'(0,\tau), {\cal D}(0,\tau)}  +
 \bigl\langle b(\widetilde{v}_{\varepsilon},\widetilde{v}_{\varepsilon},\varphi) ,\chi \bigr\rangle_{{\cal D}'(0,\tau), {\cal D}(0,\tau)} - \bigl\langle \bigl(p_{\varepsilon},div(\varphi) \bigr), \chi \bigr\rangle_{{\cal D}'(0,\tau), {\cal D}(0,\tau)} \\
\displaystyle + a(T;\widetilde{v}_{\varepsilon},\varphi\chi)
+  \bigl\langle \Psi'_{\varepsilon}(\widetilde{v}_{\varepsilon}), \varphi  \chi \bigr\rangle
\displaystyle  =
\bigl\langle (f,\varphi ), \chi \bigr\rangle_{{\cal D}'(0,\tau), {\cal D}(0,\tau)} - a(T;G_{0}\zeta ,\varphi \chi)\\
\displaystyle
- \left\langle \left(G_{0}\frac{\partial \zeta}{\partial t},\varphi \right), \chi \right\rangle_{{\cal D}'(0,\tau), {\cal D}(0,\tau)}
  -  \bigl\langle b(G_{0}\zeta ,\widetilde{v}_{\varepsilon}+G_{0}\zeta,\varphi), \chi \bigr\rangle_{{\cal D}'(0,\tau), {\cal D}(0,\tau)} \\\displaystyle- \bigl\langle b(\widetilde{v}_{\varepsilon},G_{0}\zeta ,\varphi) , \chi \bigr\rangle_{{\cal D}'(0,\tau), {\cal D}(0,\tau)}
\end{array}
\end{equation}
with the initial condition
\begin{eqnarray}\label{NS-25bisinit}
\widetilde{v}_{\varepsilon}(0, \cdot) = \widetilde{v}_{{\varepsilon} 0} \in  {\bf H}.
\end{eqnarray}

As it has been explained in Section \ref{Introduction}, a classical technique to solve such problems consists in choosing divergence free test-functions. Indeed if $\varphi \in {\cal V}_{0div}$, the term  $\bigl\langle \bigl(p_{\varepsilon},div(\varphi) \bigr), \chi \bigr\rangle_{{\cal D}'(0,\tau), {\cal D}(0,\tau)}$ vanishes and we simply get a variational problem for the fluid velocity $\widetilde{v_{\varepsilon}}$. Then the existence of  $p_{\varepsilon} \in H^{-1}\bigl(0,\tau;L^{2}_{0}(\Omega) \bigr)$ is derived via abstract results of functional analysis (see \cite{temam, simon99, galdi} for instance).

With this technique  the pressure appears as a by product of the study. In order to get better insights into the links between the velocity and pressure fields, we will follow an idea proposed by J.L.~Lions in \cite{lions} and recently used  in \cite{boukpaoli}, which consists in relaxing the divergence free condition. More precisely, we
 consider the following penalized problems $(P_{\varepsilon}^{\delta})$, $\delta >0$, $\varepsilon>0$:

\smallskip

\noindent {\bf Problem $(P_{\varepsilon}^{\delta})$}
Find
$$\widetilde{v}_{\varepsilon}^{\delta}\in L^{2} \bigl(0,\tau;{\cal V}_{0} \bigr)\cap L^{\infty}\bigl(0,\tau;{\bf L}^{2}(\Omega) \bigr) 
, \quad  \frac{ \partial \widetilde{v}_{\varepsilon}^{\delta}}{\partial t} \in L^{\frac{4}{3}} \bigl(0,\tau; ({\cal V}_{0})' \bigr)
$$
 such that, for all $ \varphi\in {\cal V}_{0}$ and for all $\chi\in {\cal D}(0,\tau)$, we have
\begin{equation}\label{NS33}
\begin{array}{ll}
\displaystyle
\left\langle \frac{d}{dt} \left( \widetilde{v}_{\varepsilon}^{\delta}, \varphi \right) , \chi\right\rangle_{{\cal D}'(0,\tau), {\cal D}(0,\tau)}  +
 \bigl\langle b(\widetilde{v}_{\varepsilon}^{\delta},\widetilde{v}_{\varepsilon}^{\delta},\varphi) ,\chi \bigr\rangle_{{\cal D}'(0,\tau), {\cal D}(0,\tau)}\\
\displaystyle +\frac{1}{2}\bigl\langle \int_{\Omega}
\widetilde{v}_{\varepsilon}^{\delta}div(\widetilde{v}_{\varepsilon}^{\delta}) \varphi \, dx , \chi \bigr\rangle_{{\cal D}'(0,\tau), {\cal D}(0,\tau)}  +\frac{1}{\delta}\bigl\langle \bigl( div(\widetilde{v}_{\varepsilon}^{\delta}), div (\varphi ) \bigr), \chi  \bigr\rangle_{{\cal D}'(0,\tau), {\cal D}(0,\tau)}
\displaystyle + a(T;\widetilde{v}_{\varepsilon}^{\delta},\varphi\chi)\\
\displaystyle
+ \bigl\langle \Psi'_{\varepsilon}(\widetilde{v}_{\varepsilon}^{\delta}) , \varphi \chi \bigr\rangle
=
\bigl\langle (f,\varphi ), \chi \bigr\rangle_{{\cal D}'(0,\tau), {\cal D}(0,\tau)} - a(T;G_{0}\zeta ,\varphi \chi)
- \left\langle \left(G_{0}\frac{\partial \zeta}{\partial t},\varphi \right), \chi \right\rangle_{{\cal D}'(0,\tau), {\cal D}(0,\tau)}
\\
\displaystyle  -  \bigl\langle b(G_{0}\zeta ,\widetilde{v}_{\varepsilon}^{\delta}+G_{0}\zeta,\varphi), \chi \bigr\rangle_{{\cal D}'(0,\tau), {\cal D}(0,\tau)} - \bigl\langle b(\widetilde{v}_{\varepsilon}^{\delta},G_{0}\zeta ,\varphi) , \chi \bigr\rangle_{{\cal D}'(0,\tau), {\cal D}(0,\tau)}
\end{array}
\end{equation}
with the initial condition
\begin{eqnarray}\label{NS33init}
\widetilde{v}_{\varepsilon}^{\delta}(0, \cdot) = \widetilde{v}_{{\varepsilon} 0}^{\delta} \in {\bf L}^{2}(\Omega)
\end{eqnarray}
and we assume that the  sequence of initial data $(\widetilde{v}_{{\varepsilon} 0}^{\delta})_{\delta>0}$ satisfies
\begin{eqnarray} \label{init2}
\displaystyle \widetilde{v}_{{\varepsilon} 0}^{\delta} \longrightarrow_{\delta \to 0} \widetilde{v}_{\varepsilon 0}
\quad \hbox{\rm strongly in } {\bf L}^{2}(\Omega).
\end{eqnarray}

Let us emphasize that the last term of the first line of (\ref{NS33})  is added for technical reasons
 (see (\ref{tech})) while the  first term of the second line is the penalty term:
$\displaystyle - \frac{1}{\delta}  div(\widetilde{v}_{\varepsilon}^{\delta})$ will play the role of
an approximate pressure (see Section \ref{approximate_pressure}). Furthermore, the approximate initial velocities $(\widetilde{v}_{{\varepsilon} 0}^{\delta})_{\varepsilon>0, \delta>0}$ and $(\widetilde{v}_{{\varepsilon} 0})_{\varepsilon >0}$ are not assumed to be more regular that $\widetilde{v}_0$.


\section{Existence result for the penalized problems $(P_{\varepsilon}^{\delta})$}
\label{penalized_problem}

We prove the existence of solutions for the system (\ref{NS33})-(\ref{NS33init}), for any $\varepsilon >0$ and $\delta>0$, by using the Galerkin method. Since ${\cal V}_{0}$ is a closed subspace of ${\bf H}^{1}(\Omega)$, it admits an Hilbertian basis $(w_{i})_{i\geq 1}$, which is  orthogonal for the inner product of  ${\bf H}^{1}(\Omega)$ and orthonormal for the inner product of ${\bf L}^{2}(\Omega)$.
Then, for all $m \ge 1$,  we look for a function $\widetilde{v}_{\varepsilon m}^{\delta}$ given by
\begin{eqnarray}\label{NS34}
\widetilde{v}_{\varepsilon m}^{\delta}(t,x) = \sum_{j=1}^{m} g_{\varepsilon j}^{\delta}(t)w_{j}(x),\quad  \forall t\in(0,\tau),\   \forall  x\in\Omega,
\end{eqnarray}
such that, for all $k \in \{1, \dots, m\}$, we have
\begin{equation}\label{NS35}
\begin{array}{ll}
\displaystyle
\left(\frac{\partial \widetilde{v}_{\varepsilon m}^{\delta}}{\partial t},w_{k}\right)+ b(\widetilde{v}_{\varepsilon m}^{\delta},\widetilde{v}_{\varepsilon m}^{\delta},w_{k})
+\frac{1}{2}\int_{\Omega}
 \widetilde{v}_{\varepsilon m}^{\delta}div(\widetilde{v}_{\varepsilon m}^{\delta}) w_{k} \, dx
+\frac{1}{\delta}\left(div(\widetilde{v}_{\varepsilon m}^{\delta}),div ( w_{k}) \right)
\\
\displaystyle + \int_{\Omega} 2 \mu(T) D(\widetilde{v}_{\varepsilon m}^{\delta}): D(w_k) \, dx
+ \int_{\Gamma_0} \ell \frac{\widetilde{v}_{\varepsilon m}^{\delta} \cdot w_k}{\sqrt{ \varepsilon^2 + |\widetilde{v}_{\varepsilon m}^{\delta}|^2}} \, dx'
=
(f,w_{k})\\
\displaystyle
- \int_{\Omega}  2 \mu(T) D(G_0 \zeta): D (w_k) \, dx
  - \left( G_{0} \frac{\partial \zeta}{\partial t},w_{k} \right)  -   b(G_{0}\zeta  , \widetilde{v}_{\varepsilon m}^{\delta} + G_{0} \zeta ,w_{k}) \\
  \displaystyle - b ( \widetilde{v}_{\varepsilon m}^{\delta}, G_0 \zeta, w_k)
\quad \mbox{\rm a.e. in } (0, \tau)
\end{array}
\end{equation}
with  the initial condition
\begin{eqnarray}\label{NS35-1}
\widetilde{v}_{\varepsilon m}^{\delta}(0, \cdot)=\widetilde{v}_{\varepsilon m 0}^{\delta}
\end{eqnarray}
where   $\widetilde{v}_{\varepsilon m 0}^{\delta}$ is defined as the orthogonal projection of $\widetilde{v}_{\varepsilon 0}^{\delta}$ in ${\bf L}^{2}(\Omega)$ on  ${\rm Span} \bigl\{w_{1}\ldots w_{m}\bigr\}$.
For all $i, j,k  \in \{1, \dots, m\}$ we denote
\begin{eqnarray*}
\displaystyle{ F_k=(f,w_{k}) -
\int_{\Omega} 2 \mu(T) D(G_0 \zeta): D(w_k) \, dx
- \left(G_{0}\frac{\partial \zeta}{\partial t},w_{k}\right)-   b(G_{0} \zeta,G_{0}\zeta ,w_{k})
 \in L^2(0,\tau) }
\end{eqnarray*}
and
\begin{eqnarray*}
A_{j,k}(T)= \int_{\Omega} 2 \mu(T) D(w_j) : D(w_k) \, dx
\in L^{\infty} (0, \tau), \quad B_{i,j,k}=b(w_{i},w_{j},w_{k}) \in \mathbb{R}.
\end{eqnarray*}
By replacing $\widetilde{v}_{\varepsilon m}^{\delta}$ by its expression (\ref{NS34}) in  equation (\ref{NS35}) and
using the orthonormality of $(w_{i})_{i \ge 1}$ in ${\bf L}^{2}(\Omega)$, we obtain
\begin{equation}\label{edo}
\begin{array}{ll}
\displaystyle{
(g_{\varepsilon k}^{\delta})'  + \sum_{i,j=1}^{m}g_{\varepsilon j}^{\delta}g_{\varepsilon i}^{\delta}B_{i,j,k}
+\frac{1}{2}\sum_{i,j=1}^{m}g_{\varepsilon i}^{\delta}g_{\varepsilon j}^{\delta}
\int_{\Omega} w_{i}div(w_{j}) w_{k} \, dx
+ \frac{1}{\delta}\sum_{j=1}^{m}g_{\varepsilon j}^{\delta}\bigl(div (w_{j}) ,div (w_{k}) \bigr)
} \\
\displaystyle{
+\sum_{j=1}^{m}g_{\varepsilon j}^{\delta}A_{j,k}(T)
+ \int_{\Gamma_0} \ell \frac{ \bigl( \sum_{j=1}^{m}g_{\varepsilon j}^{\delta}w_{j} \bigr)
\cdot w_k}{\sqrt{\varepsilon^2 + \bigl| \sum_{j=1}^{m}g_{\varepsilon j}^{\delta}w_{j} \bigr|^2}} \, dx'
 = F_{k}}\\\displaystyle{-\sum_{j=1}^{m} g_{\varepsilon j}^{\delta} b(G_{0}\zeta,w_{j},w_{k})
-\sum_{j=1}^{m}g_{\varepsilon j}^{\delta} b(w_{j},G_{0}\zeta,w_{k}) \quad \forall k \in \{1, \dots , m\}.
}
\end{array}
\end{equation}
We can rewrite this differential system as
\begin{eqnarray*}
\displaystyle{(g_{\varepsilon}^{\delta})' = {\cal G}(t,g_{\varepsilon}^{\delta}),
\quad g_{\varepsilon}^{\delta} = \bigl( g_{\varepsilon j}^{\delta} \bigr)_{1\le j \le m}}
\end{eqnarray*}
where ${\cal G}$ satisfies the assumptions of  the Caratheodory theorem (see \cite{coddington}).
 Moreover, the function ${\cal G}$ is locally Lipschitz continuous with respect its the second argument.
It follows that, for any given initial data,  the differential system (\ref{edo}) admits an unique maximal
solution  $g_{\varepsilon j}^{\delta}$ in $H^{1}(0,\tau_{m})$,  $1 \le j \le m$,
with $0<\tau_{m}\leq \tau$, which implies the  existence  of a maximal solution
$\widetilde{v}_{\varepsilon m}^{\delta} \in H^1 (0, \tau_m; {\cal V}_{0})$ to (\ref{NS35})-(\ref{NS35-1}).
In the following lemma, some a priori estimates independent of $m$, $\delta$ and  $\varepsilon$
will be established, which allow us to extend this solution to the whole interval $[0,\tau]$.

\begin{lemma}\label{lem4}
Assume that (\ref{NS14L}), (\ref{NS20}) and (\ref{G0HYPO})  hold and that  $(\widetilde{v}_{\varepsilon 0}^{\delta})_{\varepsilon>0, \delta>0}$ is a bounded sequence of  ${\bf L}^{2}(\Omega)$. The problem  (\ref{NS35})-(\ref{NS35-1}) admits a unique solution $\widetilde{v}_{\varepsilon m}^{\delta} \in H^1 \bigl(0, \tau; {\cal V}_{0} \bigr)$ which satisfies the following estimates
\begin{eqnarray}\label{NS36}
\|\widetilde{v}_{\varepsilon m}^{\delta} \|_{L^{\infty}(0,\tau; {\bf L}^{2}(\Omega))} \leq C
\end{eqnarray}
\begin{eqnarray}\label{NS37}
\|\widetilde{v}_{\varepsilon m}^{\delta}\|_{L^{2}(0,\tau; {\bf H}^1 (\Omega) )} \leq C
\end{eqnarray}
\begin{eqnarray}\label{NS37D}
\|div(\widetilde{v}_{\varepsilon m}^{\delta})\|_{L^{2}(0,\tau;L^{2}(\Omega))} \leq C\sqrt{\delta}
\end{eqnarray}
where $C$ is a  constant independent of   $m$,  $\delta$ and $\varepsilon$.
\end{lemma}

\begin{proof}
By multiplying  equation (\ref{NS35}) by $g_{\varepsilon k}^{\delta}(t)$ and adding from $k=1$ to $m$, we obtain
\begin{eqnarray*}
\begin{array}{ll}
\displaystyle{ \left(\frac{\partial \widetilde{v}_{\varepsilon m}^{\delta}}{\partial t},\widetilde{v}_{\varepsilon m}^{\delta}\right)+ b(\widetilde{v}_{\varepsilon m}^{\delta},\widetilde{v}_{\varepsilon m}^{\delta},\widetilde{v}_{\varepsilon m}^{\delta})
+\frac{1}{2}\int_{\Omega}
\widetilde{v}_{\varepsilon m}^{\delta}div(\widetilde{v}_{\varepsilon m}^{\delta})\widetilde{v}_{\varepsilon m}^{\delta} \, dx
+\frac{1}{\delta}\left(div (\widetilde{v}_{\varepsilon m}^{\delta}),div ( \widetilde{v}_{\varepsilon m}^{\delta}) \right)
} \\
\displaystyle{
+ \int_{\Omega} 2 \mu(T) D(\widetilde{v}_{\varepsilon m}^{\delta}) : D(\widetilde{v}_{\varepsilon m}^{\delta}) \, dx
+ \int_{\Gamma_0} \ell \frac{|\widetilde{v}_{\varepsilon m}^{\delta}|^{2}}{\sqrt{\varepsilon^{2}+|\widetilde{v}_{\varepsilon m}^{\delta}|^{2}}} \, dx'
= ( f ,\widetilde{v}_{\varepsilon m}^{\delta})}\\
\displaystyle
- \int_{\Omega} 2 \mu(T) D (G_{0} \zeta) : D ( \widetilde{v}_{\varepsilon m}^{\delta} ) \, dx
 - \left( G_{0} \frac{\partial \zeta}{\partial t}, \widetilde{v}_{\varepsilon m}^{\delta} \right)
  -   b(G_{0}\zeta  , \widetilde{v}_{\varepsilon m}^{\delta} + G_{0} \zeta , \widetilde{v}_{\varepsilon m}^{\delta}) \\
  \displaystyle - b ( \widetilde{v}_{\varepsilon m}^{\delta}, G_0 \zeta, \widetilde{v}_{\varepsilon m}^{\delta})
\quad \mbox{\rm a.e. in } (0, \tau_m) .
\end{array}
\end{eqnarray*}
With (\ref{formuleb}) we get
\begin{eqnarray} \label{tech}
 b(\widetilde{v}_{\varepsilon m}^{\delta},\widetilde{v}_{\varepsilon m}^{\delta},\widetilde{v}_{\varepsilon m}^{\delta})
+\frac{1}{2}\int_{\Omega}
\widetilde{v}_{\varepsilon m}^{\delta}div(\widetilde{v}_{\varepsilon m}^{\delta}) \widetilde{v}_{\varepsilon m}^{\delta} \, dx =0
\end{eqnarray}
and since  $div (G_{0}) =0$ in $\Omega$,  we have also $b(G_{0},\widetilde{v}_{\varepsilon m}^{\delta},\widetilde{v}_{\varepsilon m}^{\delta})=0$. Furthermore
since $ \ell \in L^{2}(0,\tau; {\bf L}^{2}_{+}(\Gamma_0))$, we obtain
\begin{eqnarray*}
\begin{array}{ll}
\displaystyle{ \left(\frac{\partial \widetilde{v}_{\varepsilon m}^{\delta}}{\partial t},\widetilde{v}_{\varepsilon m}^{\delta}\right)
+\frac{1}{\delta}\left(div (\widetilde{v}_{\varepsilon m}^{\delta}),div ( \widetilde{v}_{\varepsilon m}^{\delta}) \right)
+ \int_{\Omega} 2 \mu(T) D(\widetilde{v}_{\varepsilon m}^{\delta}) : D(\widetilde{v}_{\varepsilon m}^{\delta}) \, dx } \\
\displaystyle{
\le  ( f ,\widetilde{v}_{\varepsilon m}^{\delta})
- \int_{\Omega} 2 \mu(T) D (G_{0} \zeta) : D ( \widetilde{v}_{\varepsilon m}^{\delta} ) \, dx
- \left( G_{0} \frac{\partial \zeta}{\partial t}, \widetilde{v}_{\varepsilon m}^{\delta} \right) }\\
\displaystyle  -   b(G_{0}\zeta  ,  G_{0} \zeta , \widetilde{v}_{\varepsilon m}^{\delta}) - b ( \widetilde{v}_{\varepsilon m}^{\delta}, G_0 \zeta, \widetilde{v}_{\varepsilon m}^{\delta})
\quad \mbox{\rm a.e. in } (0, \tau_m) .
\end{array}
\end{eqnarray*}
Let us estimate now the terms in the right-hand side of the previous inequality.
We denote hereinafter by $K$ the constant of the continuous injection of ${\bf H}^{1}(\Omega)$ into ${\bf L}^{4}(\Omega)$.
By using  Cauchy-Schwarz's and Young's inequalities, we obtain
\begin{eqnarray*}
\left|(f,\widetilde{v}_{\varepsilon m}^{\delta})\right|
&\leq& \|f\|_{{\bf L}^{2}(\Omega)}\|\widetilde{v}_{\varepsilon m}^{\delta}\|_{{\bf L}^{2}(\Omega)}\\&\leq& \frac{1}{2}\|f\|_{{\bf L}^{2}(\Omega)}^{2}+\frac{1}{2}\|\widetilde{v}_{\varepsilon m}^{\delta}\|_{{\bf L}^{2}(\Omega)}^{2},
\end{eqnarray*}
\begin{eqnarray*}
\left| \int_{\Omega} 2 \mu(T) D(G_0 \zeta): D(\widetilde{v}_{\varepsilon m}^{\delta}) \, dx \right|
&\leq& \mu_{*}|\zeta|\|\widetilde{v}_{\varepsilon m}^{\delta}\|_{{\bf H}^{1}(\Omega)}\|G_{0}\|_{{\bf H}^{1}(\Omega)}\\&\leq& \frac{\alpha}{4}\|\widetilde{v}_{\varepsilon m}^{\delta}\|_{{\bf H}^{1}(\Omega)}^{2}+\frac{\mu_{*}^{2}}{\alpha}|\zeta|^{2} \|G_{0}\|_{{\bf H}^{1}(\Omega)}^{2},
\end{eqnarray*}
\begin{eqnarray*}
\left|\left(G_{0}\frac{\partial \zeta}{\partial t} ,\widetilde{v}_{\varepsilon m}^{\delta}\right)\right|&\leq& \left|\frac{\partial \zeta}{\partial t}\right|\|G_{0}\|_{{\bf L}^{2}(\Omega)}\|\widetilde{v}_{\varepsilon m}^{\delta}\|_{{\bf L}^{2}(\Omega)}\\&\leq& \frac{1}{2}\left|\frac{\partial \zeta}{\partial t}\right|^{2}\|G_{0}\|_{{\bf L}^{2}(\Omega)}^{2}+\frac{1}{2}\|\widetilde{v}_{\varepsilon m}^{\delta}\|_{{\bf L}^{2}(\Omega)}^{2},
\end{eqnarray*}
\begin{eqnarray*}
\left|b(G_{0}\zeta ,G_{0} \zeta ,\widetilde{v}_{\varepsilon m}^{\delta})\right|&\leq&    |\zeta|^{2} \|G_{0}\|_{{\bf L}^{4}(\Omega)}\|\nabla G_{0}\|_{{\bf L}^{4}(\Omega)}\|\widetilde{v}_{\varepsilon m}^{\delta}\|_{{\bf L}^{2}(\Omega)}\\&\leq& \frac{1}{2}\|\widetilde{v}_{\varepsilon m}^{\delta} \|_{{\bf L}^{2}(\Omega)}^{2}+\frac{K^{4}}{2}|\zeta|^{4}\|G_{0}\|^{2}_{{\bf H}^{1}(\Omega)}\|\nabla G_{0}\|^{2}_{{\bf H}^{1}(\Omega)},
\end{eqnarray*}
and
\begin{eqnarray*}
\left| b(\widetilde{v}_{\varepsilon m}^{\delta},G_{0} \zeta,\widetilde{v}_{\varepsilon m}^{\delta})\right|&\leq& |\zeta|\|\widetilde{v}_{\varepsilon m}^{\delta}\|_{{\bf L}^{4}(\Omega)}\|\nabla G_{0}\|_{{\bf L}^{4}(\Omega)}\|\widetilde{v}_{\varepsilon m}^{\delta}\|_{{\bf L}^{2}(\Omega)}\\&\leq& \frac{\alpha}{4}\|\widetilde{v}_{\varepsilon m}^{\delta}\|_{{\bf H}^{1}(\Omega)}^{2}+\frac{K^{4}}{\alpha}|\zeta|^{2}\|\nabla G_{0}\|_{{\bf H}^{1}(\Omega)}^{2}\|\widetilde{v}_{\varepsilon m}^{\delta}\|_{{\bf L}^{2}(\Omega)}^{2}.
\end{eqnarray*}
With (\ref{coercif}) and an integration
 from  $0$ to $s$, with $0<s<\tau_m$, we get
\begin{eqnarray*}
\begin{array}{ll}
\displaystyle{\frac{1}{2}\|\widetilde{v}_{\varepsilon m}^{\delta}(s)\|^{2}_{\bf{L}^{2}(\Omega)}
+\frac{1}{\delta}\int_{0}^{s}\|div(\widetilde{v}_{\varepsilon m}^{\delta})\|^{2}_{L^{2}(\Omega)}\,dt
+ \frac{\alpha}{2}\int_{0}^{s}\|\widetilde{v}_{\varepsilon m}^{\delta}\|^{2}_{{\bf H}^{1}(\Omega)}\,dt
\leq\frac{1   }{2}\|\widetilde{v}_{\varepsilon m}^{\delta}(0)\|^{2}_{{\bf L}^{2}(\Omega)}}\\
\displaystyle{+\frac{1}{2}\int_{0}^{s}\|f\|_{{\bf L}^{2}(\Omega)}^{2}\,dt +\frac{\mu_{*}^{2}}{\alpha}\|G_{0}\|_{{\bf H}^{1}(\Omega)}^{2}\int_{0}^{s}|\zeta|^{2}\,dt+\frac{1}{2}\|G_{0}\|_{{\bf L}^{2}(\Omega)}^{2}\int_{0}^{s}\left|\frac{\partial \zeta}{\partial t}\right|^{2}\,dt}\\
\displaystyle{+\frac{3}{2}\int_{0}^{s}\|\widetilde{v}_{\varepsilon m}^{\delta}\|_{{\bf L}^{2}(\Omega)}^{2}\,dt+\frac{K^{4}}{\alpha}\|\nabla G_{0}\|_{{\bf H}^{1}(\Omega)}^{2}\int_{0}^{s}|\zeta|^{2}\|\widetilde{v}_{\varepsilon m}^{\delta}\|_{{\bf L}^{2}(\Omega)}^{2}\,dt}\\\displaystyle{+\frac{K^{4}}{2}\|G_{0}\|^{2}_{{\bf H}^{1}(\Omega)}\| \nabla G_{0}\|^{2}_{{\bf H}^{1}(\Omega)}\int_{0}^{s}|\zeta|^{4}\,dt.}
\end{array}
\end{eqnarray*}
Reminding that $\widetilde{v}_{\varepsilon m 0}^{\delta}$ is defined as the orthogonal projection of $\widetilde{v}_{\varepsilon 0}^{\delta}$ in ${\bf L}^{2}(\Omega)$ on  ${\rm Span} \bigl\{w_{1}\ldots w_{m}\bigr\}$ and that the sequence $(\widetilde{v}_{\varepsilon  0}^{\delta})_{\varepsilon>0, \delta>0}$ is bounded in ${\bf L}^2(\Omega)$, we infer that there exists a constant $C_0$, independent of $\delta$ and $\varepsilon$ such that
\begin{eqnarray*}
\|\widetilde{v}_{\varepsilon m }^{\delta}(0)\|_{L^{2}(\Omega)} = \|\widetilde{v}_{\varepsilon m 0}^{\delta}\|_{L^{2}(\Omega)}  \le \|\widetilde{v}_{\varepsilon 0 }^{\delta} \|_{L^{2}(\Omega)}  \le C_0 \quad \forall m \ge 1, \ \forall \delta>0, \  \forall \varepsilon >0.
\end{eqnarray*}
It follows that
\begin{eqnarray}
\label{NS36-2}
\begin{array}{ll}
\displaystyle \frac{1}{2}\|\widetilde{v}_{\varepsilon m}^{\delta}(s)\|^{2}_{{\bf L}^{2}(\Omega)}
+\frac{1}{\delta}\int_{0}^{s}\|div(\widetilde{v}_{\varepsilon m}^{\delta})\|^{2}_{L^{2}(\Omega)}\,dt
+ \frac{\alpha}{2}\int_{0}^{s}\|\widetilde{v}_{\varepsilon m}^{\delta}\|^{2}_{{\bf H}^{1}(\Omega)}\,dt
\leq
 C_{1}\\
 \displaystyle +C_{2}\int_{0}^{s}\|\widetilde{v}_{\varepsilon m}^{\delta}\|_{{\bf L}^{2}(\Omega)}^{2}\,dt,
 \end{array}
\end{eqnarray}
where $C_{1}$ and $C_{2}$ are two constants independent of  $m$,  $\delta$ and $\varepsilon$, namely
\begin{eqnarray*}
\begin{array}{ll}
\displaystyle C_{1}=\frac{1}{2}C_0^2 +\frac{1}{2}\int_{0}^{\tau}\|f\|_{{\bf L}^{2}(\Omega)}^{2}\,dt +
\frac{\mu_{*}^{2}}{\alpha}\|G_{0}\|_{{\bf H}^{1}(\Omega)}^{2}\int_{0}^{\tau}|\zeta|^{2}\,dt
+ \frac{1}{2}\|G_{0}\|_{{\bf L}^{2}(\Omega)}^{2}\int_{0}^{\tau}\left|\frac{\partial \zeta}{\partial t}\right|^{2}\,dt \\
\displaystyle
  +\frac{   K^{4}}{2}\|G_{0}\|^{2}_{{\bf H}^{1}(\Omega)}\| \nabla G_{0}\|^{2}_{{\bf H}^{1}(\Omega)}\int_{0}^{\tau}|\zeta|^{4}\,dt
\end{array}
\end{eqnarray*}
and
\begin{eqnarray*}
\displaystyle C_{2}=\frac{3}{2}+\frac{K^{4}}{\alpha}\|\nabla G_{0}\|_{{\bf H}^{1}(\Omega)}^{2} \|\zeta\|^{2}_{L^{\infty}(0,\tau)}.
\end{eqnarray*}
With Gr\"onwall's lemma, we obtain
\begin{eqnarray}\label{estimL2}
\|\widetilde{v}_{\varepsilon m}^{\delta}(s)\|^{2}_{{\bf L}^{2}(\Omega)} \leq 2C_{1}\exp\left(2s C_{2}\right)  \le 2C_{1}\exp\left(2 \tau C_{2}\right) \quad \forall s \in [0, \tau_m).
\end{eqnarray}
With (\ref{NS34}) and (\ref{edo}) we infer that the functions $g_{\varepsilon j}^{\delta}$, $1 \le j \le m$,   admit a limit at $\tau_m$ and, by definition of the maximal solution, we may conclude  that $\tau_m=\tau$.
Now, (\ref{NS36}) follows from (\ref{estimL2}). By inserting (\ref{estimL2}) in (\ref{NS36-2}) with $s=\tau$, we obtain (\ref{NS37}) and (\ref{NS37D}).
$\Box$
\end{proof}

In the following lemma, we establish an estimate of the time derivative for the approximate velocity.

\begin{lemma} \label{lem5} Under  the same assumptions as in  Lemma \ref{lem4}, we have
\begin{eqnarray}\label{eq47}
\left\|\frac{\partial \widetilde{v}_{\varepsilon m}^{\delta}}{\partial t} \right\|_{L^{\frac{4}{3}}(0,\tau;{\cal V}_{0}')}\leq C_{\delta}
\end{eqnarray}
where $C_{\delta}$ is a constant independent of $m$ and $\varepsilon$.
\end{lemma}

\begin{proof}
Let $\varphi \in {\cal V}_{0}$. For all $m \ge 1$, we define $\varphi_{m}$ as the othogonal projection with respect to the  inner product of
${\bf H}^{1}(\Omega)$
of $\varphi$ on   ${\rm Span} \bigl\{w_{1},\ldots,w_{m} \bigr\}$.
With (\ref{NS35}) we get
\begin{eqnarray*}
\begin{array}{ll}
\displaystyle{   \left(\frac{ \partial \widetilde{v}_{\varepsilon m}^{\delta}}{\partial t} ,\varphi_{m}\right)
= - b(\widetilde{v}_{\varepsilon m}^{\delta},\widetilde{v}_{\varepsilon m}^{\delta},\varphi_{m})-
\frac{1}{2}\int_{\Omega}
\widetilde{v}_{\varepsilon m}^{\delta}div (\widetilde{v}_{\varepsilon m}^{\delta}) \varphi_{m} \, dx }
-\frac{1}{\delta}\bigl(div(\widetilde{v}_{\varepsilon m}^{\delta}),div (\varphi_{m}) \bigr)
\\
\displaystyle{- \int_{\Omega} 2 \mu(T) D(\widetilde{v}_{\varepsilon m}^{\delta}) : D(\varphi_m) \, dx
- \int_{\Gamma_0} \ell \frac{\widetilde{v}_{\varepsilon m}^{\delta} \cdot \varphi_m}{\sqrt{\varepsilon^2 + |\widetilde{v}_{\varepsilon m}^{\delta}|^2}} \, dx'
+
(f, \varphi_m)}\\\displaystyle{
- \int_{\Omega}  2 \mu(T) D(G_0 \zeta): D (\varphi_m) \, dx
  - \left( G_{0} \frac{\partial \zeta}{\partial t}, \varphi_m \right) -   b(G_{0}\zeta  , \widetilde{v}_{\varepsilon m}^{\delta} + G_{0} \zeta , \varphi_m) }\\\displaystyle{- b ( \widetilde{v}_{\varepsilon m}^{\delta}, G_0 \zeta, \varphi_m)
\quad \mbox{\rm a.e. in } (0, \tau).}
\end{array}
\end{eqnarray*}
We estimate all the terms in the right hand side of the previous equality, we obtain
\begin{eqnarray*}
\begin{array}{ll}
\displaystyle{\left|\left( \frac{\partial \widetilde{v}_{\varepsilon m}^{\delta}}{\partial t},\varphi_{m}\right)\right|
\leq
 \left(\|\widetilde{v}_{\varepsilon m}^{\delta}\|_{{\bf L}^{3}(\Omega)}\|\nabla\widetilde{v}_{\varepsilon m}^{\delta}\|_{{\bf L}^{2}(\Omega)}
+\frac{1}{2}\|\widetilde{v}_{\varepsilon m}^{\delta}\|_{{\bf L}^{3}(\Omega)} \bigl\|div (\widetilde{v}_{\varepsilon m}^{\delta}) \bigr \|_{ L^{2}(\Omega)}\right)\|\varphi_{m}\|_{{\bf L}^{6}(\Omega)}
}
\\
\displaystyle{
+\frac{1}{\delta}\bigl\|div (\widetilde{v}_{\varepsilon m}^{\delta})\bigr\|_{L^{2}(\Omega)}\bigl\|div (\varphi_{m}) \bigr\|_{L^{2}(\Omega)}
+\mu_{*}\| \widetilde{v}_{\varepsilon m}^{\delta}\|_{{\bf H}^{1}(\Omega)}\| \varphi_{m}\|_{{\bf H}^{1}(\Omega)}
}
\\
\displaystyle{+\|\ell\|_{{\bf L}^{2}(\Gamma_0)}\|\varphi_{m}\|_{{\bf L}^{2}(\Gamma_0)}
+ \|f\|_{{\bf L}^{2}(\Omega)}\|\varphi_{m}\|_{{\bf L}^{2}(\Omega)}+\mu_{*}|\zeta|\|G_{0}\|_{{\bf H}^{1}(\Omega)}\|\varphi_{m}\|_{{\bf H}^{1}(\Omega)} } \\
\displaystyle
+   \left|\frac{\partial\zeta}{\partial t}\right|\|G_{0}\|_{{\bf L}^{2}(\Omega)}\|\varphi_{m}\|_{{\bf L}^{2}(\Omega)}
+   |\zeta|\|G_{0}\|_{{\bf L}^{4}(\Omega)}\|\nabla\widetilde{v}_{\varepsilon m}^{\delta}+\nabla G_{0}\|_{{\bf L}^{2}(\Omega)}\|\varphi_{m}\|_{{\bf L}^{4}(\Omega)}
\\
\displaystyle +|\zeta|\|\widetilde{v}_{\varepsilon m}^{\delta}\|_{{\bf L}^{4}(\Omega)}\|\nabla G_{0}\|_{{\bf L}^{2}(\Omega)}\|\varphi_{m}\|_{{\bf L}^{4}(\Omega)} \quad \mbox{\rm a.e. in } (0, \tau).
\end{array}
\end{eqnarray*}
By using the classical inequality
$$\|u\|_{L^{3}(\Omega)}\leq \|u\|_{L^{2}(\Omega)}^{\frac{1}{2}}\|u\|_{L^{6}(\Omega)}^{\frac{1}{2}} \quad \forall u\in L^{6}(\Omega) \cap L^2 (\Omega) $$
and the injection of  ${\bf H}^{1}(\Omega)$ in ${\bf L}^{6}(\Omega)$,  we infer that there exists a constant $c$, independent of $m$,  $\delta$ and $\varepsilon$,  such that
\begin{eqnarray*}
\|\widetilde{v}_{\varepsilon m}^{\delta}\|_{{\bf L}^{3}(\Omega)}\|\nabla\,\widetilde{v}_{\varepsilon m}^{\delta}\|_{{\bf L}^{2}(\Omega)}\|\varphi_{m}\|_{{\bf L}^{6}(\Omega)}
\leq c \|\widetilde{v}_{\varepsilon m}^{\delta}\|_{{\bf L}^{2}(\Omega)}^{\frac{1}{2}}
\|\widetilde{v}_{\varepsilon m}^{\delta}\|_{{\bf H}^{1}(\Omega)}^{\frac{3}{2}}
\|\varphi_{m}\|_{{\bf H}^{1}(\Omega)}.
\end{eqnarray*}
As $(w_{j})_{j\geq 1}$ is an orthogonal family of  ${\bf L}^{2}(\Omega)$ and $\varphi_{m}$ is the orthogonal
projection with respect to the  inner product of
${\bf H}^{1}(\Omega)$
of $\varphi$ on   ${\rm Span} \bigl\{w_{1},\ldots,w_{m} \bigr\}$,
 we have $\|\varphi_{m}\|_{{\bf H}^{1}(\Omega)}\leq \|\varphi\|_{{\bf H}^{1}(\Omega)}$ and
\begin{eqnarray*}
\left( \frac{ \widetilde{v}_{\varepsilon m}^{\delta}}{\partial t} ,\varphi_{m}\right)=\left( \frac{\partial \widetilde{v}_{\varepsilon m}^{\delta}}{\partial t} ,\varphi_{k} \right) \quad \forall k \ge m.
\end{eqnarray*}
Since $(w_{j})_{j\geq 1}$ is an Hilbertian basis of ${\cal V}_0$, the sequence $(\varphi_k)_{k \ge 1}$ converges strongly to $\varphi$ in ${\bf H}^1(\Omega)$ and we get
\begin{eqnarray*}
\left( \frac{ \widetilde{v}_{\varepsilon m}^{\delta}}{\partial t} ,\varphi_{m}\right)=\left( \frac{\partial \widetilde{v}_{\varepsilon m}^{\delta}}{\partial t} ,\varphi \right) .
\end{eqnarray*}
Then, we obtain
\begin{eqnarray*}
\begin{array}{ll}
\displaystyle{
\left|\left( \frac{\partial \widetilde{v}_{\varepsilon m}^{\delta}}{\partial t},\varphi\right)\right|
\leq
  \left( \frac{ \sqrt{3} }{2}+1 \right) c \|\widetilde{v}_{\varepsilon m}^{\delta}\|_{{\bf L}^{2}(\Omega)}^{\frac{1}{2}}
\| \widetilde{v}_{\varepsilon m}^{\delta}\|_{{\bf H}^{1}(\Omega)}^{\frac{3}{2}}
\|\varphi\|_{{\bf H}^{1}(\Omega)}
}

 \\
\displaystyle{+\frac{\sqrt{3}}{\delta} \bigl\|div (\widetilde{v}_{\varepsilon m}^{\delta}) \bigr\|_{L^{2}(\Omega)}
\|\varphi\|_{{\bf H}^{1}(\Omega)}
+ \mu_{*}\|\widetilde{v}_{\varepsilon m}^{\delta}\|_{{\bf H}^{1}(\Omega)}\|\varphi\|_{{\bf H}^{1}(\Omega)}
+ \tilde c \| \ell\|_{{\bf L}^{2}(\Gamma_0)}
\|\varphi\|_{{\bf H}^{1}(\Omega)}
} \\
\displaystyle{+\left(\|f\|_{{\bf L}^{2}(\Omega)}+\mu_{*}|\zeta|\|G_{0}\|_{{\bf H}^{1}(\Omega)}+   \left|\frac{\partial\zeta}{\partial t}\right|\|G_{0}\|_{{\bf L}^{2}(\Omega)}\right)\|\varphi\|_{{\bf H}^{1}(\Omega)}} \\
\displaystyle{+\left(K^{2}|\zeta|\|G_{0}\|_{{\bf H}^{1}(\Omega)}\|\widetilde{v}_{\varepsilon m}^{\delta}+ G_{0}\zeta\|_{{\bf H}^{1}(\Omega)}+K^{2}|\zeta|\|\widetilde{v}_{\varepsilon m}^{\delta}\|_{{\bf H}^{1}(\Omega)}\|G_{0}\|_{{\bf H}^{1}(\Omega)}\right)\|\varphi\|_{{\bf H}^{1}(\Omega)} \ \mbox{\rm a.e. in $(0, \tau)$,} }
\end{array}
\end{eqnarray*}
where $\tilde c$ is the norm of the trace operator $\gamma_0: {\bf H}^{1}(\Omega) \to {\bf L}^2 (\Gamma_0)$.
Hence
\begin{eqnarray*}
\begin{array}{ll}
\displaystyle{
\left\| \frac{\partial \widetilde{v}_{\varepsilon m}^{\delta}}{\partial t} \right\|_{{\cal V}_{0}'}
 \leq
 \left( \frac{\sqrt{3}}{2} +1 \right) c\|\widetilde{v}_{\varepsilon m}^{\delta}\|_{{\bf L}^{2}(\Omega)}^{\frac{1}{2}}\|\widetilde{v}_{\varepsilon m}^{\delta}\|_{{\bf H}^{1}(\Omega)}^{\frac{3}{2}}
 +\frac{\sqrt{3} }{\delta}\bigl\|div (\widetilde{v}_{\varepsilon m}^{\delta}) \bigr\|_{L^{2}(\Omega)}
+ \mu_{*}\|\widetilde{v}_{\varepsilon m}^{\delta}\|_{{\bf H}^{1}(\Omega)} }
\\
\displaystyle{
+ \tilde c \| \ell \|_{{\bf L}^{2}(\Gamma_0)}+\|f\|_{{\bf L}^{2}(\Omega)} +\mu_{*}|\zeta|\|G_{0}\|_{{\bf H}^{1}(\Omega)}+   \left|\frac{\partial\zeta}{\partial t}\right|\|G_{0}\|_{{\bf L}^{2}(\Omega)}}\\
\displaystyle{+K^{2}|\zeta|\|G_{0}\|_{{\bf H}^{1}(\Omega)}\|\widetilde{v}_{\varepsilon m}^{\delta}+ G_{0}\zeta\|_{{\bf H}^{1}(\Omega)}+K^{2}|\zeta|\|\widetilde{v}_{\varepsilon m}^{\delta}\|_{{\bf H}^{1}(\Omega)}\|G_{0}\|_{{\bf H}^{1}(\Omega)}  \quad \mbox{\rm a.e. in } (0, \tau).}
\end{array}
\end{eqnarray*}
Observing that
\begin{eqnarray*}
\int_{0}^{\tau}\left[\|\widetilde{v}_{\varepsilon m}^{\delta}\|_{{\bf L}^{2}(\Omega)}^{\frac{1}{2}}\|\widetilde{v}_{\varepsilon m}^{\delta}\|_{{\bf H}^{1}(\Omega)}^{\frac{3}{2}} \right]^{\frac{4}{3}}\,dt&=&  \int_{0}^{\tau}\|\widetilde{v}_{\varepsilon m}^{\delta}\|_{{\bf L}^{2}(\Omega)}^{\frac{2}{3}}\|\widetilde{v}_{\varepsilon m}^{\delta}\|_{{\bf H}^{1}(\Omega)}^{2}\,dt \\
&\leq&    \|\widetilde{v}_{\varepsilon m}^{\delta}\|_{L^{\infty}(0,\tau;{\bf L}^{2}(\Omega))}^{\frac{2}{3}}\|\widetilde{v}_{\varepsilon m}^{\delta}\|_{L^{2}(0,\tau;{\bf H}^{1}(\Omega))}^{2},
\end{eqnarray*}
we infer from the estimates of Lemma \ref{lem4}  that there exists a constant $C_{\delta}>0$, independent of $m$ and $\varepsilon$,  such that
\begin{eqnarray*}
\displaystyle{   \int_{0}^{\tau} \left\| \frac{\partial \widetilde{v}_{\varepsilon m}^{\delta}}{\partial t}  \right\|_{{\cal V}_{0}'}^{\frac{4}{3}}\,dt \leq C_{\delta} }
\end{eqnarray*}
which concludes the proof. $\Box$
\end{proof}

In order to pass to the limit as $m$ tends to $+ \infty$, we will use also the following Lemma.

\begin{lemma}\label{lem6}
Let $\varepsilon >0$ and $\ell \in L^{2}\bigl(0,\tau; {\bf L}^{2}_+(\Gamma_0) \bigr) \cap L^{\infty}\bigl(0,\tau; {\bf L}^{\infty}_+(\Gamma_0) \bigr)$. Then the mapping $\Psi'_{\varepsilon}$ is Lipschitz continuous from $L^{2}\bigl(0,\tau; {\bf L}^{2}(\Gamma_0) \bigr)$ to $L^{2}\bigl(0,\tau; {\bf L}^{2}(\Gamma_0) \bigr)$.
\end{lemma}

\begin{proof}
Let us recall that, for all $u \in L^2 \bigl(0,\tau; {\bf L}^2(\Gamma_0) \bigr)$,
$\Psi'_{\varepsilon}(u) \in \bigl( L^2 \bigl(0,\tau; {\bf L}^2(\Gamma_0) \bigr) \bigr)'= L^2 \bigl(0,\tau; {\bf L}^2(\Gamma_0) \bigr) $ is defined by
\begin{eqnarray*}
\langle \Psi'_{\varepsilon}(u),w \rangle = \int_0^{\tau} \int_{\Gamma_0} \ell \frac{u \cdot w}{\sqrt{\varepsilon^{2} + |u|^{2}}}\,dx' dt \quad \forall  w \in L^2 \bigl(0,\tau; {\bf L}^2(\Gamma_0) \bigr)
\end{eqnarray*}
where $\langle \cdot, \cdot \rangle$ denotes the inner product in $L^2 \bigl(0,\tau; {\bf L}^2(\Gamma_0) \bigr) $, i.e.
\begin{eqnarray*}
\Psi'_{\varepsilon}(u) = \ell \frac{u }{\sqrt{\varepsilon^{2} + |u|^{2}}} \quad \forall u \in L^2 \bigl(0,\tau; {\bf L}^2(\Gamma_0) \bigr) .
\end{eqnarray*}
But the mapping
\begin{eqnarray*}
h_{\varepsilon}:  \left\{
\begin{array}{ll}
{\mathbb R}^d \to {\mathbb R}^d \\
\displaystyle u \mapsto  \frac{u }{\sqrt{\varepsilon^{2} + |u|^{2}}}
\end{array} \right.
\end{eqnarray*}
is Fr\'echet differentiable on ${\mathbb R}^d$ and
\begin{eqnarray*}
{\rm Jac}(h_{\varepsilon}) (u) = \left( \frac{\partial h_{\varepsilon i}}{\partial x_j} (u) \right)_{1 \le i,j \le d} = \left( \frac{\delta_{i,j}}{\sqrt{\varepsilon^{2} + |u|^{2}}} - \frac{u_i u_j}{\bigl( \varepsilon^{2} + |u|^{2} \bigr)^{\frac{3}{2}}} \right)_{1 \le i,j \le d}
\end{eqnarray*}
where $\delta_{i,j} =1$ if $i=j$ and $\delta_{i,j} =0$ if $i \not=j$.
It follows that
\begin{eqnarray*}
\left| \frac{\partial h_{\varepsilon i}}{\partial x_j} (u) \right| \le \frac{2}{\varepsilon} \quad \forall i, j \in \{1, \dots, d\}, \  \forall u \in {\mathbb R}^d
\end{eqnarray*}
and $h_{\varepsilon}$ is Lipschitz continuous on ${\mathbb R}^d$. Since $\ell \in L^{\infty}\bigl(0,\tau; {\bf L}^{\infty}_+(\Gamma_0) \bigr)$, we infer that $\Psi_{\varepsilon}'$ is Lipschitz continuous from $L^{2}\bigl(0,\tau; {\bf L}^{2}_+(\Gamma_0) \bigr)$ into $L^{2}\bigl(0,\tau; {\bf L}^{2}_+(\Gamma_0) \bigr)$.
$\Box$
\end{proof}

Now, by using  the estimates obtained in Lemma \ref{lem4} and Lemma \ref{lem5} combined with compactness arguments, we can prove the following existence result for the penalized problems $(P_{\varepsilon}^{\delta})$.

\begin{theorem}\label{lemme2.3}
Let $\varepsilon>0$ and  $\delta >0$. Assume
that (\ref{NS14L}), (\ref{NS20}) and (\ref{G0HYPO})  hold and that  $(\widetilde{v}_{\varepsilon 0}^{\delta})_{\varepsilon>0, \delta>0}$ is a bounded sequence of   ${\bf L}^{2}(\Omega)$. Then, there exists a subsequence of $(\widetilde{v}_{\varepsilon m}^{\delta})_{m \ge 1}$, still denoted $(\widetilde{v}_{\varepsilon m}^{\delta})_{m \ge 1}$, such that
\begin{eqnarray}\label{conver2}
\displaystyle{
\widetilde{v}_{\varepsilon m}^{\delta}\rightharpoonup \widetilde{v}_{\varepsilon}^{\delta}\quad\mbox{weakly star in $L^{\infty}\bigl(0,\tau;{\bf L}^{2}(\Omega) \bigr)$}}
\end{eqnarray}
\begin{eqnarray}\label{conver1}
\displaystyle{
\widetilde{v}_{\varepsilon m}^{\delta}\rightharpoonup \widetilde{v}_{\varepsilon}^{\delta}\quad\mbox{weakly in $ L^{2}(0,\tau;{\cal V}_{0}) $ }}
\end{eqnarray}
and $\widetilde{v}_{\varepsilon}^{\delta}$ is solution of $(P_{\varepsilon}^{\delta})$. Furthermore $\displaystyle \frac{\partial \widetilde{v}_{\varepsilon}^{\delta}}{\partial t}$ belongs to $L^{\frac{4}{3}} (0,\tau; {\cal V}_{0}' )$.
\end{theorem}

\begin{proof}
The convergences (\ref{conver2})-(\ref{conver1})  follow immediately from the a priori estimates (\ref{NS36})-(\ref{NS37}) obtained in  Lemma \ref{lem4}.
From the estimate (\ref{eq47}) obtained in Lemma \ref{lem5}, we infer that, possibly extracting another subsequence still denoted  denoted  $(\widetilde{v}_{\varepsilon m}^{\delta})_{m\geq 1}$, we have
\begin{eqnarray}\label{conver3}
\frac{\partial \widetilde{v}_{\varepsilon m}^{\delta} }{\partial t} \rightharpoonup \frac{\partial \widetilde{v}_{\varepsilon}^{\delta}}{\partial t} \quad\mbox{weakly in $L^{\frac{4}{3}}(0,\tau;{\cal V}_{0}')$.}
\end{eqnarray}
By using  Aubin's lemma \cite{temam} and the convergences (\ref{conver1}) and (\ref{conver3}), with $X_{0}={\cal V}_{0}$, $X={\bf L}^{4}(\Omega)$ and $X_{1}={\cal V}_{0}'$ we obtain
\begin{eqnarray*}
\widetilde{v}_{\varepsilon m}^{\delta}\rightarrow \widetilde{v}_{\varepsilon}^{\delta}\quad\mbox{strongly in $L^{2}(0,\tau;{\bf L}^{4}(\Omega))$.}
\end{eqnarray*}
We may use again  Aubin's lemma  with $X_{0}={\cal V}_{0}$, $X={\bf H}^{s}(\Omega)$ and $X_{1}={\cal V}_{0}'$ with $ \frac{1}{2}<s<1$: the embedding of $X_{0}$ into $X$ is compact, so we obtain
\begin{eqnarray*}
\widetilde{v}_{\varepsilon m}^{\delta}\rightarrow \widetilde{v}_{\varepsilon}^{\delta}\quad\mbox{strongly in $L^{2}(0,\tau;{\bf H}^{s}(\Omega))$}.
\end{eqnarray*}
Then, with  trace theorem \cite{lions2}, we infer that
\begin{eqnarray*}
\widetilde{v}_{\varepsilon m}^{\delta}\rightarrow \widetilde{v}_{\varepsilon}^{\delta}\quad\mbox{strongly in $L^{2}(0,\tau;{\bf L}^{2}(\Gamma_0))$}
\end{eqnarray*}
where we identify here the functions $\widetilde{v}_{\varepsilon m}^{\delta}$ and $\widetilde{v}_{\varepsilon}^{\delta}$ with their trace on $\Gamma_0$.

Now,  using  (\ref{conver2})-(\ref{conver3}) and  Simon's lemma \cite{simon87} and  possibly extracting another subsequence, still  denoted  $(\widetilde{v}_{\varepsilon m}^{\delta})_{m\geq 1}$,   we obtain
\begin{eqnarray}\label{lemsimonc3}
\widetilde{v}_{\varepsilon m}^{\delta}\rightarrow \widetilde{v}_{\varepsilon}^{\delta}\quad\mbox{strongly in ${\cal C}^{0}(0,\tau;H)$,}
\end{eqnarray}
for any Banach space $H$  such that ${\bf L}^{2}(\Omega)\subset H \subset {\cal V}_{0}'$ with continuous injections and compact embedding of ${\bf L}^{2}(\Omega)$ into $H$.

Let $\chi\in {\cal D}(0,\tau )$ and $\varphi\in {\cal V}_{0}$. For all $m \ge 1$ we define again $\varphi_{m}$ as the othogonal projection with respect to the  inner product of
${\bf H}^{1}(\Omega)$
of $\varphi$ on   ${\rm Span} \bigl\{w_{1},\ldots,w_{m} \bigr\}$.
With  (\ref{NS35}) we have
\begin{eqnarray*}
\begin{array}{ll}
\displaystyle{ \int_0^{\tau}  \left[  \left(\frac{\partial \widetilde{v}_{\varepsilon m}^{\delta}}{\partial t},\varphi_{m} \right)   + b(\widetilde{v}_{\varepsilon m}^{\delta},\widetilde{v}_{\varepsilon m}^{\delta},\varphi_{m})
 +\frac{1}{2}\int_{\Omega}
\widetilde{v}_{\varepsilon m}^{\delta}div(\widetilde{v}_{\varepsilon m}^{\delta})\varphi_{m} \, dx\right] \chi \, dt}
\\
\displaystyle{
+ \frac{1}{\delta}\int_{0}^{\tau} \left(div(\widetilde{v}_{\varepsilon m}^{\delta}),div (\varphi_{m})\chi \right)\,dt

+a(T;\widetilde{v}_{\varepsilon m}^{\delta},\varphi_{m}\chi)+ \left\langle \Psi'_{\varepsilon}(\widetilde{v}_{\varepsilon m}^{\delta}),\varphi_{m}\chi\right\rangle
=
\int_0^{\tau} (f,\varphi_{m})\chi \, dt } \\
\displaystyle-  a(T;G_{0} \zeta ,\varphi_{m}\chi)  - \int_0^{\tau} \left[ \left(G_{0}\frac{\partial \zeta}{\partial t},\varphi_{m} \right)
+  b(G_{0} \zeta ,\widetilde{v}_{\varepsilon m}^{\delta}+G_{0}\zeta,\varphi_{m})
+  b(\widetilde{v}_{\varepsilon m}^{\delta},G_{0} \zeta ,\varphi_{m} )  \right] \chi \, dt.
\end{array}
\end{eqnarray*}
With an integration by parts of the first term we get
\begin{eqnarray*}
\begin{array}{ll}
\displaystyle{ \int_0^{\tau}   \left( \widetilde{v}_{\varepsilon m}^{\delta},\varphi_{m} \right)  \frac{\partial \chi}{\partial t} \, dt
 + \int_0^{\tau} \left[ b(\widetilde{v}_{\varepsilon m}^{\delta},\widetilde{v}_{\varepsilon m}^{\delta},\varphi_{m}) +\frac{1}{2}\int_{\Omega} \widetilde{v}_{\varepsilon m}^{\delta}div(\widetilde{v}_{\varepsilon m}^{\delta}) \varphi_{m} \, dx\right] \chi \, dt
}
\\
\displaystyle{
+ \frac{1}{\delta}\int_{0}^{\tau}\left(div(\widetilde{v}_{\varepsilon m}^{\delta}),div (\varphi_{m}) \chi\right)\,dt
+a(T;\widetilde{v}_{\varepsilon m}^{\delta},\varphi_{m}\chi)+ \left\langle \Psi'_{\varepsilon}(\widetilde{v}_{\varepsilon m}^{\delta}),\varphi_{m}\chi\right\rangle
=
\int_0^{\tau} (f,\varphi_{m})\chi \, dt  } \\
\displaystyle -  a(T;G_{0} \zeta ,\varphi_{m}\chi)- \int_0^{\tau} \left[ \left(G_{0}\frac{\partial \zeta}{\partial t},\varphi_{m} \right)
+  b(G_{0} \zeta ,\widetilde{v}_{\varepsilon m}^{\delta}+G_{0}\zeta,\varphi_{m})
+  b(\widetilde{v}_{\varepsilon m}^{\delta},G_{0} \zeta ,\varphi_{m} )
\right] \chi \, dt.
\end{array}
\end{eqnarray*}
Reminding that $(\varphi_m)_{m \ge 1}$ converges strongly to $\varphi$ in ${\bf H}^1(\Omega)$ and using Lemma \ref{lem6},
 we can pass to the limit in all the terms and we obtain
\begin{eqnarray*}
\begin{array}{ll}
\displaystyle{ \int_0^{\tau}   \left( \widetilde{v}_{\varepsilon }^{\delta},\varphi \right)  \frac{\partial \chi}{\partial t}  \, dt  + \int_0^{\tau} \left[ b(\widetilde{v}_{\varepsilon }^{\delta},\widetilde{v}_{\varepsilon }^{\delta},\varphi)
+\frac{1}{2}\int_{\Omega} \widetilde{v}_{\varepsilon }^{\delta}div(\widetilde{v}_{\varepsilon }^{\delta}) \varphi \, dx
+ \frac{1}{\delta}\left(div(\widetilde{v}_{\varepsilon }^{\delta}),div (\varphi ) \right)
\right] \chi \, dt
}
\\
\displaystyle{+a(T;\widetilde{v}_{\varepsilon }^{\delta},\varphi \chi)+ \left\langle \Psi'_{\varepsilon}(\widetilde{v}_{\varepsilon }^{\delta}),\varphi \chi\right\rangle
=
\int_0^{\tau} (f,\varphi )\chi \, dt -  a(T;G_{0} \zeta ,\varphi \chi) } \\
\displaystyle - \int_0^{\tau} \left[ \left(G_{0}\frac{\partial \zeta}{\partial t},\varphi  \right)
+  b(G_{0} \zeta ,\widetilde{v}_{\varepsilon }^{\delta}+G_{0}\zeta,\varphi )
+  b(\widetilde{v}_{\varepsilon }^{\delta},G_{0} \zeta ,\varphi  )
 \right] \chi \, dt.
\end{array}
\end{eqnarray*}
which gives (\ref{NS33}). It remains to chek that the initial condition (\ref{NS33init}) is satisfied. Indeed, with (\ref{lemsimonc3}), we have
\begin{eqnarray*}
\widetilde{v}_{\varepsilon m}^{\delta} (0) \rightarrow \widetilde{v}_{\varepsilon}^{\delta}(0)  \quad \mbox{strongly in $H$}
\end{eqnarray*}
with ${\bf L}^{2}(\Omega)\subset H \subset {\cal V}_{0}'$ and we have also
\begin{eqnarray*}
\widetilde{v}_{\varepsilon m}^{\delta} (0)  = \widetilde{v}_{\varepsilon m 0}^{\delta}\rightarrow \widetilde{v}_{\varepsilon 0}^{\delta}  \quad \mbox{strongly in ${\bf L}^2 (\Omega)$.}
\end{eqnarray*}
Hence $\widetilde{v}_{\varepsilon}^{\delta}(0) = \widetilde{v}_{\varepsilon 0}^{\delta}$.
$\Box$
\end{proof}

\section{Properties of the approximate pressure}
\label{approximate_pressure}

For any $\varepsilon >0$ and $\delta >0$ we define and approximate pressure $p^{\delta}_{\varepsilon} \in L^{2} \bigl(0,\tau;L^{2}(\Omega) \bigr)$ by
\begin{eqnarray}\label{pnote}
p^{\delta}_{\varepsilon}=-\frac{1}{\delta}div (\widetilde{v}^{\delta}_{\varepsilon})
\end{eqnarray}
where $\widetilde{v}_{\varepsilon}^{\delta}$ is the solution of the penalized problem $(P_{\varepsilon}^{\delta})$ obtained in the previous Section. From  (\ref{NS33}) we get
\begin{equation}\label{eqP3}
\begin{array}{ll}
\displaystyle
\left\langle \frac{d}{dt} \left( \widetilde{v}_{\varepsilon}^{\delta}, \varphi \right) , \chi\right\rangle_{{\cal D}'(0,\tau), {\cal D}(0,\tau)}  +
 \bigl\langle b(\widetilde{v}_{\varepsilon}^{\delta},\widetilde{v}_{\varepsilon}^{\delta},\varphi) ,\chi \bigr\rangle_{{\cal D}'(0,\tau), {\cal D}(0,\tau)}\\
\displaystyle
+\frac{1}{2}\bigl\langle \int_{\Omega}
 \widetilde{v}_{\varepsilon}^{\delta}div(\widetilde{v}_{\varepsilon}^{\delta}) \varphi \, dx , \chi \bigr\rangle_{{\cal D}'(0,\tau), {\cal D}(0,\tau)} - \bigl\langle \bigl( p_{\varepsilon}^{\delta} , div(\varphi) \bigr), \chi  \bigr\rangle_{{\cal D}'(0,\tau), {\cal D}(0,\tau)}
 + a(T;\widetilde{v}_{\varepsilon}^{\delta},\varphi\chi)\\

\displaystyle  + \bigl\langle \Psi'_{\varepsilon}(\widetilde{v}_{\varepsilon}^{\delta}), \varphi  \chi \bigr\rangle
 =
\bigl\langle (f,\varphi ), \chi \bigr\rangle_{{\cal D}'(0,\tau), {\cal D}(0,\tau)} - a(T;G_{0}\zeta ,\varphi \chi)\\
\displaystyle
- \left\langle \left(G_{0}\frac{\partial \zeta}{\partial t},\varphi \right), \chi \right\rangle_{{\cal D}'(0,\tau), {\cal D}(0,\tau)}
  -  \bigl\langle b(G_{0}\zeta ,\widetilde{v}_{\varepsilon}^{\delta}+G_{0}\zeta,\varphi), \chi \bigr\rangle_{{\cal D}'(0,\tau), {\cal D}(0,\tau)} \\
  \displaystyle - \bigl\langle b(\widetilde{v}_{\varepsilon}^{\delta},G_{0}\zeta ,\varphi) , \chi \bigr\rangle_{{\cal D}'(0,\tau), {\cal D}(0,\tau)}, \quad \forall \varphi\in {\cal V}_{0}, \  \forall \chi\in {\cal D}(0,\tau).
\end{array}
\end{equation}
Furthermore, with Green's formula, we obtain
\begin{eqnarray}
\label{p=0}
\int_{\Omega} p_{\varepsilon}^{\delta}\,dx = -\frac{1}{\delta}\int_{\partial\Omega}\widetilde{v}_{\varepsilon}^{\delta} \cdot n \,dx = 0 \quad \mbox{a.e. in $(0, \tau)$}
\end{eqnarray}
and, with (\ref{NS37D}) and (\ref{conver1}), we have
\begin{eqnarray*}
\bigl\| p_{\varepsilon}^{\delta}\bigr\|_{L^{2} (0,\tau;L^{2}(\Omega))} \le\frac{C}{\sqrt{\delta}}
\end{eqnarray*}
where $C$ is a constant independent of $\delta$ and $\varepsilon$.  Unfortunately this last estimate does not allow us to pass to the limit in the term $\bigl\langle \bigl( p_{\varepsilon}^{\delta}, div(\varphi) \bigr), \chi  \bigr\rangle_{{\cal D}'(0,\tau), {\cal D}(0,\tau)}$ as $\delta$ tends to zero. So we will establish an estimate independent of $\varepsilon$ and $\delta$ by using the same kind of technique as in \cite{boukpaoli}).

\begin{lemma} \label{lem-pression}
Under the same assumptions as in Lemma \ref{lem4},
there exists a constant $C$, independent of $\delta$ and $\varepsilon$, such that
\begin{eqnarray}
\label{estim-pression}
\bigl\| p_{\varepsilon}^{\delta} \bigr\|_{H^{-1} (0,\tau;L^{2}(\Omega))} \le C.
\end{eqnarray}
\end{lemma}



\begin{proof}
Let $\chi \in {\cal D}(0,\tau)$ and $w\in L^{2}_{0}(\Omega)$. Then there exists  $\varphi  \in {\bf H}^1_0(\Omega)$ such that
\begin{eqnarray*}
div(\varphi) = w \quad \mbox{in $\Omega$}
\end{eqnarray*}
and $\varphi=P(w)$ where $P$ is a linear continuous operator from $L^{2}_{0}(\Omega)$ into ${\bf  H}^{1}_{0}(\Omega)$ (see \cite{lions}).
With an integration by parts of the first term of (\ref{eqP3}), we get
\begin{eqnarray*}
\begin{array}{ll}
\displaystyle
\int_0^{\tau} \bigl( p_{\varepsilon}^{\delta}, div(\varphi) \bigr) \chi \, dt
=
 - \int_0^{\tau} \left( \widetilde{v}_{\varepsilon}^{\delta}, \varphi \right) \frac{\partial  \chi}{\partial t} \, dt
+ \int_0^{\tau} \left[ b(\widetilde{v}_{\varepsilon}^{\delta},\widetilde{v}_{\varepsilon}^{\delta},\varphi)
+\frac{1}{2} \int_{\Omega}
 \widetilde{v}_{\varepsilon}^{\delta}div(\widetilde{v}_{\varepsilon}^{\delta}) \varphi \, dx \right] \chi  \, dt
 \\
\displaystyle
+ a(T;\widetilde{v}_{\varepsilon}^{\delta},\varphi\chi)
- \int_0^{\tau}  (f,\varphi ) \chi \, dt  + a(T;G_{0}\zeta ,\varphi \chi) \\
\displaystyle
+ \int_0^{\tau} \left[
\left(G_{0}\frac{\partial \zeta}{\partial t},\varphi \right) +  b(G_{0}\zeta ,\widetilde{v}_{\varepsilon}^{\delta}+G_{0}\zeta,\varphi) +  b(\widetilde{v}_{\varepsilon}^{\delta},G_{0}\zeta ,\varphi) \right] \chi \, dt
\end{array}
\end{eqnarray*}
Let us denote   $\varphi\chi=  \eta$ and recall that $K$ is the constant of the continuous injection of ${\bf H}^1 (\Omega)$ into ${\bf L}^4(\Omega)$. We get
\begin{eqnarray*}
\begin{array}{ll}
\displaystyle{
\left| \int_{0}^{\tau}
 \bigl[ b(\widetilde{v}_{\varepsilon }^{\delta},\widetilde{v}_{\varepsilon }^{\delta},\varphi )
+\frac{1}{2} \int_{\Omega} \widetilde{v}_{\varepsilon }^{\delta}div(\widetilde{v}_{\varepsilon }^{\delta}) \varphi \, dx
\bigr] \chi \,dt \right|\leq}\\\displaystyle{
\leq \left( 1 + \frac{\sqrt{3}}{2} \right)
 \|\widetilde{v}_{\varepsilon }^{\delta}\|_{L^{2}(0,\tau;{\bf L}^{4}(\Omega))}
 \|\widetilde{v}_{\varepsilon }^{\delta}\|_{L^{2}(0,\tau;{\bf H}^{1}(\Omega))}
\|\eta\|_{L^{\infty}(0,\tau;{\bf L}^4(\Omega))}}
\\
\displaystyle \le   K^2 \left( 1 + \frac{\sqrt{3}}{2} \right)  \|\widetilde{v}_{\varepsilon }^{\delta}\|^2_{L^{2}(0,\tau; {\bf H}^{1} (\Omega) )}\|\eta\|_{L^{\infty}(0,\tau;{\bf H}^{1} (\Omega)) }.
\end{array}
\end{eqnarray*}
Then we obtain
\begin{eqnarray*}
\begin{array}{ll}
\displaystyle{
\left| \int_{0}^{\tau}(p^{\delta}_{\varepsilon},w) \chi \,dt\right|
\leq
  \|\widetilde{v}_{\varepsilon }^{\delta}\|_{L^{2}(0,\tau;{\bf L}^{2}(\Omega))}
\|\frac{\partial \eta}{\partial t}\|_{L^{2}(0,\tau;{\bf L}^{2}(\Omega))} }
\\
\displaystyle{
+ K^2   \left( 1 + \frac{\sqrt{3}}{2} \right)
 \|\widetilde{v}_{\varepsilon }^{\delta}\|^2_{L^{2}(0,\tau; {\bf H}^{1} (\Omega) )}
\|\eta\|_{L^{\infty}(0,\tau;{\bf H}^{1} (\Omega)) }
 }
\\
\displaystyle{
+ \mu_{*}\sqrt{\tau}
\|\widetilde{v}_{\varepsilon }^{\delta}\|_{L^{2}(0,\tau; {\bf H}^{1} (\Omega) )}
\|\eta\|_{L^{\infty}(0,\tau; {\bf H}^{1} (\Omega))}
 +\sqrt{\tau}\|f\|_{L^{2}(0,\tau;{\bf L}^{2}(\Omega))}\|\eta\|_{L^{\infty}(0,\tau;{\bf L}^{2}(\Omega))}}
\\
\displaystyle{
+\mu_{*}\|G_{0}\|_{{\bf H}^{1}(\Omega)}\|\zeta\|_{L^{1}(0,\tau)}\|\eta\|_{L^{\infty}(0,\tau; {\bf H}^{1}(\Omega))}
+\|G_{0}\|_{{\bf L}^{2}(\Omega)}\left\|\frac{\partial\zeta}{\partial t}\right\|_{L^{1}(0,\tau)}\|\eta\|_{L^{\infty}(0,\tau; {\bf L}^{2}(\Omega))}
}\\
\displaystyle{
+ K^2  \|G_{0}\|_{{\bf H}^{1}(\Omega)} \|\zeta\|_{L^{2}(0,\tau)} \|\widetilde{v}_{\varepsilon}^{\delta}+G_{0}\zeta\|_{L^{2}(0,\tau; {\bf H}^{1}(\Omega))}\|\eta\|_{L^{\infty}(0,\tau; {\bf H}^{1}(\Omega))}
}\\
\displaystyle{+ K^2  \|G_{0}\|_{{\bf H}^{1}(\Omega)}\|\zeta\|_{L^{2}(0,\tau)}\|\widetilde{v}_{\varepsilon}^{\delta}\|_{L^{2}(0,\tau;{\bf H}^1(\Omega) )}\|\eta\|_{L^{\infty}(0,\tau; {\bf H}^{1} (\Omega))}.}
\end{array}
\end{eqnarray*}
By using the continuity of the operator $P$ and the continuous injection of $H^1(0, \tau)$ into $L^{\infty}(0,\tau)$, we have
\begin{eqnarray*}
\|\eta\|_{L^{\infty}(0,\tau; {\bf H}^{1} (\Omega))} = \|  \chi\|_{L^{\infty}(0, \tau) }
\bigl\|P (w) \bigr\|_{{\bf H}^{1}(\Omega)}
\leq C  \|  \chi\|_{H^{1}(0, \tau) } \|w\|_{ {\bf L}^{2}(\Omega)} \le  C  \|\eta \|_{H^{1} (0,\tau;{\bf  L}^{2} (\Omega))},
\end{eqnarray*}
where $C$ is a constant independent of $\delta$ and $\varepsilon$. With the estimates (\ref{NS36}) and (\ref{NS37}),  we
infer that $\widetilde{v}_{\varepsilon}^{\delta}$ is bounded in $L^{2}(0,\tau;{\bf H}^1(\Omega) ) \cap L^{\infty}(0,\tau; {\bf L}^{2}(\Omega))$ independently of $\delta$ and $\varepsilon$, and we
 obtain
\begin{eqnarray}\label{eqP4}
\left|\int_{0}^{\tau}(p^{\delta}_{\varepsilon},w) \chi\,dt\right| \leq C\| w \chi\|_{H^{1} (0,\tau; L^{2} (\Omega))}\quad \forall  w\in L^{2}_{0}(\Omega), \  \forall \chi \in {\cal D}(0, \tau)
\end{eqnarray}
where we denote again by $C$ a constant independent of $\delta$ and $\varepsilon$.

Let now $\widetilde{w} \in  L^{2}(\Omega)$. We can apply (\ref{eqP4}) with
\begin{eqnarray*}
w= \widetilde{w}-\frac{1}{mes \Omega}\int_{\Omega}\widetilde{w} \,dx.
\end{eqnarray*}
Indeed, $w \in L^{2}_{0}(\Omega)$.  Furthermore, with (\ref{p=0}), we have
\begin{eqnarray*}
\begin{array}{ll}
\displaystyle \left|\int_{0}^{\tau} \left( p^{\delta}_{\varepsilon},\widetilde{w}-\frac{1}{mes \Omega}\int_{\Omega}\widetilde{w}\,dx \right) \chi \,dt\right|=\\\displaystyle
  =
\left|\int_{0}^{\tau}(p^{\delta}_{\varepsilon},\widetilde{w}) \chi \,dt-\frac{1}{mes \Omega}\int_{0}^{\tau}\left(\int_{\Omega}p_{\varepsilon}^{\delta}\,dx\right)\left(\int_{\Omega}\widetilde{w} \,dx\right) \chi \,dt\right|
 =  \left|\int_{0}^{\tau}(p^{\delta}_{\varepsilon},\widetilde{w}) \chi \,dt \right|.
 \end{array}
\end{eqnarray*}
Observing that $ \|w\|_{L^{2}(\Omega)} \leq \|\widetilde{w}\|_{L^{2}(\Omega)} $,
we obtain
\begin{eqnarray}\label{PL2}
\left|\int_{0}^{\tau}(p^{\delta}_{\varepsilon}, w ) \chi \,dt\right|  = \left|\int_{0}^{\tau}(p^{\delta}_{\varepsilon}, \widetilde{w}) \chi \,dt\right| \leq C\| \widetilde{w} \chi \|_{H^{1}(0,\tau; L^{2}(\Omega))}, \  \forall  \widetilde{w} \in L^{2}(\Omega) , \  \forall \chi \in {\cal D}(0, \tau).
\end{eqnarray}
Then the density of ${\cal D}(0,\tau) \otimes L^2(\Omega)$ into $H^{1}_{0}(0,\tau;L^{2} (\Omega))$ allows us to conclude.
$\Box$
\end{proof}

\section{Existence results for the problems $(P_{\varepsilon})$ and $(P)$}
\label{fin}

Now we can pass to the  limit in the penalized problems $(P_{\varepsilon}^{\delta})$ when $\delta$ tends to zero.

\begin{theorem}\label{prop1}
Let $\varepsilon>0$ and assume that $(\widetilde{v}_{\varepsilon 0}^{\delta})_{\varepsilon>0, \delta>0}$ is a bounded sequence of   ${\bf L}^{2}(\Omega)$. Assume moreover that (\ref{NS14L}), (\ref{NS20}), (\ref{G0HYPO}) and (\ref{init2}) hold. Then, there exists a subsequence of $(\widetilde{v}_{\varepsilon }^{\delta}, p_{\varepsilon}^{\delta})_{\delta >0}$, still denoted $(\widetilde{v}_{\varepsilon }^{\delta}, p_{\varepsilon}^{\delta})_{\delta >0}$, such that
\begin{eqnarray}\label{conver2bis}
\displaystyle{
\widetilde{v}_{\varepsilon }^{\delta}\rightharpoonup \widetilde{v}_{\varepsilon} \quad\mbox{weakly star in $L^{\infty}\bigl(0,\tau;{\bf L}^{2}(\Omega) \bigr)$}}
\end{eqnarray}
\begin{eqnarray}\label{conver1bis}
\displaystyle{
\widetilde{v}_{\varepsilon }^{\delta}\rightharpoonup \widetilde{v}_{\varepsilon}\quad\mbox{weakly in $ L^{2}(0,\tau;{\cal V}_{0}) $ }}
\end{eqnarray}
\begin{eqnarray}\label{conver3bis}
\displaystyle{
\widetilde{p}_{\varepsilon }^{\delta}\rightharpoonup \widetilde{p}_{\varepsilon} \quad\mbox{weakly  in $H^{-1}\bigl(0,\tau;{ L}^{2}_0 (\Omega) \bigr)$}}
\end{eqnarray}
and $(\widetilde{v}_{\varepsilon}, p_{\varepsilon})$ is solution of $(P_{\varepsilon})$. Furthermore $\displaystyle \frac{\partial \widetilde{v}_{\varepsilon}}{\partial t}$ belongs to $L^{\frac{4}{3}}\bigl(0,\tau;({\cal V}_{0div})' \bigr)$.
\end{theorem}

\begin{proof}
Observing that the estimates obtained in Lemma \ref{lem4} are independent of $m$, $\delta$ and $\varepsilon$, we infer that the sequence $(\widetilde{v}_{\varepsilon }^{\delta}, p_{\varepsilon}^{\delta})_{\delta >0}$ is bounded in $ L^{2}(0,\tau;{\cal V}_{0}) \cap L^{\infty}\bigl(0,\tau;{\bf L}^{2}(\Omega) \bigr) $. Moreover, with Proposition \ref{lem-pression}, the sequence $(p_{\varepsilon}^{\delta})_{\delta >0}$ is bounded in $  H^{-1}\bigl(0,\tau;{\bf L}^{2}(\Omega) \bigr)$ and the convergences (\ref{conver1bis})-(\ref{conver2bis})-(\ref{conver3bis}) follow immediately.

From (\ref{NS37D}) we infer that
\begin{eqnarray*}
\bigl\| div(\widetilde{v}_{\varepsilon }^{\delta}) \bigr\|_{L^{2}(0,\tau; L^{2}(\Omega))} \le C \sqrt{\delta}
\end{eqnarray*}
with a constant $C$ independent of $\delta$ and $\varepsilon$.
Thus
\begin{eqnarray*}
div(\widetilde{v}_{\varepsilon }^{\delta})\rightarrow 0\quad\mbox{strongly in $L^{2} \bigl(0,\tau; L^{2}(\Omega) \bigr)$.}
\end{eqnarray*}
Finally we can obtain an estimate of $\displaystyle \frac{\partial \widetilde{v}_{\varepsilon}^{\delta}}{\partial t}$ in  $L^{\frac{4}{3}}\bigl(0,\tau;({\cal V}_{0div})' \bigr)$ by using the same kind of computations as in Lemma \ref{lem5}. Indeed, let $\varphi \in {\cal V}_{0div}$ and $\chi \in {\cal D} (0, \tau)$. With (\ref{eqP3}) we get
\begin{eqnarray*}
\begin{array}{ll}
\displaystyle
\left\langle \frac{d}{dt} \left( \widetilde{v}_{\varepsilon}^{\delta}, \varphi \right) , \chi\right\rangle_{{\cal D}'(0,\tau), {\cal D}(0,\tau)}  = \int_0^{\tau}  \int_{\Omega} \frac{\partial \widetilde{v}_{\varepsilon}^{\delta}}{\partial t}  \varphi   \chi \, dx dt =
-  \bigl\langle b(\widetilde{v}_{\varepsilon}^{\delta},\widetilde{v}_{\varepsilon}^{\delta},\varphi) ,\chi \bigr\rangle_{{\cal D}'(0,\tau), {\cal D}(0,\tau)} \\
\displaystyle
- \frac{1}{2}\bigl\langle \int_{\Omega}
 \widetilde{v}_{\varepsilon}^{\delta}div(\widetilde{v}_{\varepsilon}^{\delta}) \varphi \, dx , \chi \bigr\rangle_{{\cal D}'(0,\tau), {\cal D}(0,\tau)}
\displaystyle -  a(T;\widetilde{v}_{\varepsilon}^{\delta},\varphi\chi)
-  \bigl\langle \Psi'_{\varepsilon}(\widetilde{v}_{\varepsilon}^{\delta}) , \varphi  \chi \bigr\rangle \\
\displaystyle
+ \bigl\langle (f,\varphi ), \chi \bigr\rangle_{{\cal D}'(0,\tau), {\cal D}(0,\tau)}   - a(T;G_{0}\zeta ,\varphi \chi)
\displaystyle  - \left\langle \left(G_{0}\frac{\partial \zeta}{\partial t},\varphi \right), \chi \right\rangle_{{\cal D}'(0,\tau), {\cal D}(0,\tau)} \\
\displaystyle -  \bigl\langle b(G_{0}\zeta ,\widetilde{v}_{\varepsilon}^{\delta}+G_{0}\zeta,\varphi), \chi \bigr\rangle_{{\cal D}'(0,\tau), {\cal D}(0,\tau)} - \bigl\langle b(\widetilde{v}_{\varepsilon}^{\delta},G_{0}\zeta ,\varphi) , \chi \bigr\rangle_{{\cal D}'(0,\tau), {\cal D}(0,\tau)}.
\end{array}
\end{eqnarray*}
We can estimate all the terms in the right hand side of the previous equality and we obtain
\begin{eqnarray*}
\begin{array}{ll}
\displaystyle{
\left|
 \int_0^{\tau}  \int_{\Omega} \frac{\partial \widetilde{v}_{\varepsilon}^{\delta}}{\partial t}  \varphi  \chi \, dx dt
\right|
\leq
\left( \frac{ \sqrt{3} }{2}+1 \right) c \int_0^\tau  \|\widetilde{v}_{\varepsilon }^{\delta}\|_{{\bf L}^{2}(\Omega)}^{\frac{1}{2}}
\| \widetilde{v}_{\varepsilon }^{\delta}\|_{{\bf H}^{1}(\Omega)}^{\frac{3}{2}}
\|\varphi\|_{{\bf H}^{1}(\Omega)} |\chi| \, dt
}
\\
\displaystyle{
 + \mu_{*} \int_0^\tau \|\widetilde{v}_{\varepsilon }^{\delta}\|_{{\bf H}^{1}(\Omega)}\|\varphi\|_{{\bf H}^{1}(\Omega)} |\chi| \, dt
    }
\displaystyle{+ \tilde c \int_0^\tau \| \ell\|_{{\bf L}^{2}(\Gamma_0)}
\|\varphi\|_{{\bf H}^{1}(\Omega)} |\chi| \, dt } \\
\displaystyle{+ \int_0^{\tau} \left[ \|f\|_{{\bf L}^{2}(\Omega)}+\mu_{*}|\zeta|\|G_{0}\|_{{\bf H}^{1}(\Omega)}+   \left|\frac{\partial\zeta}{\partial t}\right|\|G_{0}\|_{{\bf L}^{2}(\Omega)}\right] \|\varphi\|_{{\bf H}^{1}(\Omega)} |\chi| \, dt } \\
\displaystyle{+ \int_0^{\tau} \left[ K^{2}|\zeta|\|G_{0}\|_{{\bf H}^{1}(\Omega)}\|\widetilde{v}_{\varepsilon }^{\delta}+ G_{0}\zeta\|_{{\bf H}^{1}(\Omega)}+K^{2}|\zeta|\|\widetilde{v}_{\varepsilon }^{\delta}\|_{{\bf H}^{1}(\Omega)}\|G_{0}\|_{{\bf H}^{1}(\Omega)}\right] \|\varphi\|_{{\bf H}^{1}(\Omega)} |\chi| \, dt .}
\end{array}
\end{eqnarray*}
Observing that
\begin{eqnarray*}
\begin{array}{ll}
\displaystyle
\int_0^{\tau } \|\widetilde{v}_{\varepsilon }^{\delta}\|_{{\bf L}^{2}(\Omega)}^{\frac{1}{2}}
\| \widetilde{v}_{\varepsilon }^{\delta}\|_{{\bf H}^{1}(\Omega)}^{\frac{3}{2}}
\|\varphi\|_{{\bf H}^{1}(\Omega)} |\chi| \, dt
\le\\
\displaystyle \le \|\widetilde{v}_{\varepsilon }^{\delta}\|_{L^{\infty}(0, \tau; {\bf L}^{2}(\Omega)) }^{\frac{1}{2}}
\left(  \int_0^\tau  \| \widetilde{v}_{\varepsilon }^{\delta}\|_{{\bf H}^{1}(\Omega)}^{2} \, dt  \right)^{\frac{3}{4} }
\|\varphi\|_{{\bf H}^{1}(\Omega)} \|\chi \|_{L^4(0,\tau)} \\
\displaystyle  \le \|\widetilde{v}_{\varepsilon }^{\delta}\|_{L^{\infty}(0, \tau; {\bf L}^{2}(\Omega)) }^{\frac{1}{2}}
\|\widetilde{v}_{\varepsilon }^{\delta}\|_{L^{2}(0, \tau; {\bf H}^{1}(\Omega)) }^{\frac{3}{2}}
\|\varphi \chi\|_{L^4(0, \tau; {\bf H}^{1}(\Omega))}
\end{array}
\end{eqnarray*}
and reminding that $(\widetilde{v}_{\varepsilon }^{\delta})_{\delta>0}$ is bounded in $L^{2}\bigl(0,\tau;{\bf H}^1(\Omega) \bigr) \cap L^{\infty}\bigl(0,\tau;{\bf L}^{2}(\Omega) \bigr)$ independently of  $\delta$ and $\varepsilon$, we infer that
\begin{eqnarray}\label{eq47div}
\left\| \frac{\partial \widetilde{v}_{\varepsilon}^{\delta}}{\partial t} \right\|_{L^{\frac{4}{3}}(0,\tau;({\cal V}_{0div})' )}\leq C
\end{eqnarray}
with a constant $C$ independent of $\delta$ and $\varepsilon$.

It follows that, possibly extracting another subsequence still denoted $(\widetilde{v}_{\varepsilon }^{\delta})_{\delta>0}$,  we have
\begin{eqnarray}
\label{convp}
\frac{\partial \widetilde{v}_{\varepsilon}^{\delta}}{\partial t}  \rightharpoonup \frac{\partial \widetilde{v}_{\varepsilon}}{\partial t} \quad \mbox{weakly in $L^{\frac{4}{3}} \bigl(0,\tau;({\cal V}_{0div})' \bigr)$.}
\end{eqnarray}
By using Aubin's lemma, with $X_{0}={\cal V}_{0}$, $X={\bf L}^4(\Omega)$  and $X_{1}=({\cal V}_{0div })'$, we obtain
\begin{eqnarray*}
\widetilde{v}_{\varepsilon}^{\delta}\rightarrow \widetilde{v}_{\varepsilon}\quad\mbox{strongly in $L^{2}\bigl(0,\tau;{\bf L}^{4}(\Omega) \bigr)$}
\end{eqnarray*}
and, with $X_{0}={\cal V}_{0}$, $X={\bf H}^s(\Omega)$, $\displaystyle \frac{1}{2} <s<1$,   and $X_{1}=({\cal V}_{0div })'$,
\begin{eqnarray*}
\widetilde{v}_{\varepsilon}^{\delta}\rightarrow \widetilde{v}_{\varepsilon} \quad\mbox{strongly in $ L^{2} \bigl(0,\tau;{\bf H}^{s}(\Omega) \bigr)$.}
\end{eqnarray*}
Hence
\begin{eqnarray*}
\widetilde{v}_{\varepsilon}^{\delta}\rightarrow \widetilde{v}_{\varepsilon} \quad\mbox{strongly in $ L^{2} \bigl(0,\tau;{\bf L}^{2}(\Gamma_0) \bigr)$.}
\end{eqnarray*}
Finally, using  (\ref{conver2bis})-(\ref{convp}) and  Simon's lemma, and  possibly extracting another subsequence still  denoted  $(\widetilde{v}_{\varepsilon }^{\delta})_{\delta >0}$,   we obtain
\begin{eqnarray*}
\widetilde{v}_{\varepsilon }^{\delta}\rightarrow \widetilde{v}_{\varepsilon} \quad\mbox{strongly in ${\cal C}^{0}(0,\tau;H)$}
\end{eqnarray*}
for any Banach space $H$  such that ${\bf L}^{2}(\Omega)\subset H \subset ({\cal V}_{0div})'$ with continuous injections and compact  embedding of ${\bf L}^{2}(\Omega)$ into $H$.

With all these convergences and with the assumption (\ref{init2}), we can pass to the limit in (\ref{NS33}) and (\ref{NS33init}) by the same techniques as in Theorem \ref{lemme2.3} and we get (\ref{NS-25bis}) and (\ref{NS-25bisinit}).
$\Box$
\end{proof}

Now, observing that $\Psi_{\varepsilon}$ is convex, we obtain that
\begin{eqnarray*}
\Psi_{\varepsilon}(\widetilde{v}_{\varepsilon} + \varphi \chi ) - \Psi_{\varepsilon}(\widetilde{v}_{\varepsilon})
\ge  \bigl\langle \Psi'_{\varepsilon}(\widetilde{v}_{\varepsilon}), \varphi  \chi \bigr\rangle \quad \forall \varphi \in {\cal V}_0, \ \forall \chi \in {\cal D}(0, \tau)
\end{eqnarray*}
and in (\ref{NS-25bis}) we get
\begin{equation} \label{NS-25ter}
\begin{array}{ll}
\displaystyle
\left\langle \frac{d}{dt} \left( \widetilde{v}_{\varepsilon}, \varphi \right) , \chi\right\rangle_{{\cal D}'(0,\tau), {\cal D}(0,\tau)}  +
 \bigl\langle b(\widetilde{v}_{\varepsilon},\widetilde{v}_{\varepsilon},\varphi) ,\chi \bigr\rangle_{{\cal D}'(0,\tau), {\cal D}(0,\tau)} - \bigl\langle \bigl(p_{\varepsilon},div(\varphi) \bigr), \chi \bigr\rangle_{{\cal D}'(0,\tau), {\cal D}(0,\tau)}
  \\
\displaystyle
+ a(T;\widetilde{v}_{\varepsilon},\varphi\chi)
+   \Psi_{\varepsilon}(\widetilde{v}_{\varepsilon}+ \varphi  \chi)
-  \Psi_{\varepsilon}(\widetilde{v}_{\varepsilon})
\displaystyle  \ge
\bigl\langle (f,\varphi ), \chi \bigr\rangle_{{\cal D}'(0,\tau), {\cal D}(0,\tau)} - a(T;G_{0}\zeta ,\varphi \chi)
\\
\displaystyle - \left\langle \left(G_{0}\frac{\partial \zeta}{\partial t},\varphi \right), \chi \right\rangle_{{\cal D}'(0,\tau), {\cal D}(0,\tau)}
 -  \bigl\langle b(G_{0}\zeta ,\widetilde{v}_{\varepsilon}+G_{0}\zeta,\varphi), \chi \bigr\rangle_{{\cal D}'(0,\tau), {\cal D}(0,\tau)} \\
\displaystyle - \bigl\langle b(\widetilde{v}_{\varepsilon},G_{0}\zeta ,\varphi) , \chi \bigr\rangle_{{\cal D}'(0,\tau), {\cal D}(0,\tau)}
\end{array}
\end{equation}
for all $\varphi\in {\cal V}_{0}$ and for all $ \chi\in {\cal D}(0,\tau)$, with the initial condition
\begin{eqnarray}
\label{NS-25terinit}
\widetilde{v}_{\varepsilon}(0, \cdot) = \widetilde{v}_{{\varepsilon} 0} .
\end{eqnarray}
In order to pass to the limit as $\varepsilon$ tends to zero in the previous inequality, we use the following lemma.

\begin{lemma} \label{lem8}
Let $(w_{\varepsilon})_{\varepsilon >0}$ be a sequence of $L^2 \bigl( 0, \tau; {\bf L}^2(\Gamma_0) \bigr)$ and $w \in L^2 \bigl( 0, \tau; {\bf L}^2(\Gamma_0) \bigr)$ such that  $(w_{\varepsilon})_{\varepsilon >0}$ converges strongly to $w$ in $L^2 \bigl( 0, \tau; {\bf L}^2(\Gamma_0) \bigr)$. Then
$\displaystyle \lim_{\varepsilon \to 0} \Psi_{\varepsilon}( w_{\varepsilon}) = \Psi(w)$.
\end{lemma}

\begin{proof}
Let $\varepsilon >0$. By definition of $\Psi_ {\varepsilon}$ and $\Psi$, we have
\begin{eqnarray*}
\Psi_{\varepsilon}( w_{\varepsilon}) - \Psi (w) = \int_0^{\tau} \int_{\Gamma_0} \ell \bigl( |w_{\varepsilon}| - |w| \bigr) \, dx' dt
+ \int_0^{\tau} \int_{\Gamma_0} \ell \bigl( \sqrt{ \varepsilon^2 + |w_{\varepsilon}|^2} - |w_{\varepsilon}| \bigr) \, dx' dt.
\end{eqnarray*}
It follows that
\begin{eqnarray*}
 \bigl| \Psi_{\varepsilon}( w_{\varepsilon}) - \Psi (w)  \bigr| & \le  &
\int_0^{\tau} \int_{\Gamma_0} \ell \bigl| |w_{\varepsilon}| - |w| \bigr| \, dx' dt
+ \int_0^{\tau} \int_{\Gamma_0} \ell \varepsilon  \, dx' dt  \\
& \le &  \|\ell\|_{L^2(0, \tau; {\bf L}^2(\Gamma_0))} \bigl( \| w_{\varepsilon} - w \|_{L^2(0, \tau; {\bf L}^2(\Gamma_0))} + \varepsilon \sqrt{\tau {\rm meas}(\Gamma_0)}  \bigr)
\end{eqnarray*}
which allows us to conclude. $\Box$
\end{proof}


Now we can prove that problem $(P)$ admits a solution.

\begin{theorem}\label{thm3}
Assume that $(\widetilde{v}_{\varepsilon 0}^{\delta})_{\varepsilon>0, \delta>0}$ is a bounded sequence of   ${\bf L}^{2}(\Omega)$. Assume moreover that (\ref{NS14L}), (\ref{NS20}), (\ref{G0HYPO}) and (\ref{init1}) hold.
Then, there exists a subsequence of $(\widetilde{v}_{\varepsilon }, p_{\varepsilon})_{\varepsilon >0}$, still denoted $(\widetilde{v}_{\varepsilon }, p_{\varepsilon})_{\varepsilon >0}$ such that
\begin{eqnarray}\label{conver2ter}
\displaystyle{
\widetilde{v}_{\varepsilon } \rightharpoonup \widetilde{v} \quad\mbox{weakly star in $L^{\infty}\bigl(0,\tau;{\bf L}^{2}(\Omega) \bigr)$}}
\end{eqnarray}
\begin{eqnarray}\label{conver1ter}
\displaystyle{
\widetilde{v}_{\varepsilon } \rightharpoonup \widetilde{v} \quad\mbox{weakly in $ L^{2}(0,\tau;{\cal V}_{0}) $ }}
\end{eqnarray}
\begin{eqnarray}\label{conver3ter}
\displaystyle{
\widetilde{p}_{\varepsilon } \rightharpoonup \widetilde{p}  \quad\mbox{weakly  in $H^{-1}\bigl(0,\tau;{ L}^{2}_0(\Omega) \bigr)$}}
\end{eqnarray}
and $(\widetilde{v} , p )$ is solution of $(P )$. Furthermore $\displaystyle \frac{\partial \widetilde{v}}{\partial t}$ belongs to $L^{\frac{4}{3}}\bigl(0,\tau;({\cal V}_{0div})' \bigr)$.
\end{theorem}

\begin{proof}
 Recalling that the estimates (\ref{NS36})-(\ref{NS37}) are independent of $m$, $\delta$ and $\varepsilon$, we deduce that $(\widetilde{v}_{\varepsilon})_{\varepsilon >0}$ is bounded in $L^{2}\bigl(0,\tau ;{\bf H}^1(\Omega) \bigr) \cap L^{\infty} \bigl(0,\tau;{\bf L}^{2}(\Omega) \bigr)$. Moreover, since the estimate (\ref{estim-pression}) is independent of $\delta$ and $\varepsilon$, the sequence $(p_{\varepsilon})_{\varepsilon >0}$ is bounded in $H^{-1} \bigl(0,\tau; L^2(\Omega) \bigr)$ and we may infer the convergences (\ref{conver2ter})-(\ref{conver1ter})-(\ref{conver3ter}).
Furthermore  the estimate (\ref{eq47div}) implies that that $\displaystyle \left( \frac{\partial \widetilde{v}_{\varepsilon}}{\partial t} \right)_{\varepsilon >0}$ is bounded in  $L^{\frac{4}{3}}\bigl(0,\tau; ({\cal V}_{0div})' \bigr)$. Hence, possibly extracting another subsequence still denoted $(\widetilde{v}_{\varepsilon})_{\varepsilon >0}$, we have
\begin{eqnarray*}
\frac{\partial \widetilde{v}_{\varepsilon} }{\partial t}  \rightharpoonup \frac{\partial \widetilde{v} }{\partial t} \quad \mbox{weakly in $L^{\frac{4}{3}} \bigl(0,\tau;({\cal V}_{0div})' \bigr)$}
\end{eqnarray*}
and with the same arguments as in the previous Theorem, we get
\begin{eqnarray*}
\widetilde{v}_{\varepsilon} \rightarrow \widetilde{v} \quad\mbox{strongly in $L^{2}\bigl(0,\tau;{\bf L}^{4}(\Omega) \bigr)$,}
\end{eqnarray*}
\begin{eqnarray*}
\widetilde{v}_{\varepsilon} \rightarrow \widetilde{v} \quad\mbox{strongly in $ L^{2} \bigl(0,\tau;{\bf L}^{2}(\Gamma_0) \bigr)$,}
\end{eqnarray*}
and
\begin{eqnarray*}
\widetilde{v}_{\varepsilon }^{\delta}\rightarrow \widetilde{v}_{\varepsilon} \quad\mbox{strongly in ${\cal C}^{0}(0,\tau;H)$,}
\end{eqnarray*}
for any Banach space $H$  such that ${\bf L}^{2}(\Omega)\subset H \subset ({\cal V}_{0div})'$ with continuous injections and compact  embedding of ${\bf L}^{2}(\Omega)$ into $H$.

With all these convergences and  the assumption (\ref{init1}), we can pass to the limit in (\ref{NS-25ter}) and (\ref{NS-25terinit}) by the same techniques as in Theorem \ref{lemme2.3} and Theorem \ref{prop1} and we get (\ref{NS-25}) and (\ref{NS-25init}).
$\Box$
\end{proof}

\end{document}